\documentclass[12pt]{amsart}
\usepackage{amsmath, amssymb, amsthm, latexsym}
\usepackage{tikz-cd}
\usepackage{enumitem}
\usepackage{amssymb,amscd,amsmath}
\usepackage{latexsym}
\usepackage{ dsfont }
\usepackage[center]{caption}
\usepackage{tikz}
\usepackage{tikz}
\usepackage{mathrsfs}
\usetikzlibrary{decorations.markings}
\newcounter{braid}
\newcounter{strands}

\DeclareMathAlphabet{\bsf}{OT1}{cmss}{bx}{n}

\pgfkeyssetvalue{/tikz/braid height}{1cm}
\pgfkeyssetvalue{/tikz/braid width}{1cm}
\pgfkeyssetvalue{/tikz/braid start}{(0,0)}
\pgfkeyssetvalue{/tikz/braid colour}{black}
\pgfkeys{/tikz/strands/.code={\setcounter{strands}{#1}}}

\makeatletter
\def\cross{%
  \@ifnextchar^{\message{Got sup}\cross@sup}{\cross@sub}}

\def\cross@sup^#1_#2{\render@cross{#2}{#1}}

\def\cross@sub_#1{\@ifnextchar^{\cross@@sub{#1}}{\render@cross{#1}{1}}}

\def\cross@@sub#1^#2{\render@cross{#1}{#2}}

\def\render@cross#1#2{
  \def\strand{#1}
  \def\crossing{#2}
  \pgfmathsetmacro{\cross@y}{-\value{braid}*\braid@h}
  \pgfmathtruncatemacro{\nextstrand}{#1+1}
  \foreach \thread in {1,...,\value{strands}}
  {
    \pgfmathsetmacro{\strand@x}{\thread * \braid@w}
    \ifnum\thread=\strand
    \pgfmathsetmacro{\over@x}{\strand * \braid@w + .5*(1 - \crossing) * \braid@w}
    \pgfmathsetmacro{\under@x}{\strand * \braid@w + .5*(1 + \crossing) * \braid@w}
    \draw[braid] \pgfkeysvalueof{/tikz/braid start} +(\under@x pt,\cross@y pt) to[out=-90,in=90] +(\over@x pt,\cross@y pt -\braid@h);
    \draw[braid] \pgfkeysvalueof{/tikz/braid start} +(\over@x pt,\cross@y pt) to[out=-90,in=90] +(\under@x pt,\cross@y pt -\braid@h);
    \else
    \ifnum\thread=\nextstrand
    \else
     \draw[braid] \pgfkeysvalueof{/tikz/braid start} ++(\strand@x pt,\cross@y pt) -- ++(0,-\braid@h);
    \fi
   \fi
  }
  \stepcounter{braid}
}

\tikzset{braid/.style={double=\pgfkeysvalueof{/tikz/braid colour},double distance=1pt,line width=2pt,white}}

\newcommand{\braid}[2][]{%
  \begingroup
  \pgfkeys{/tikz/strands=2}
  \tikzset{#1}
  \pgfkeysgetvalue{/tikz/braid width}{\braid@w}
  \pgfkeysgetvalue{/tikz/braid height}{\braid@h}
  \setcounter{braid}{0}
  \let\sigma=\cross
  #2
  \endgroup
}
\makeatother

\input xypic
\newtheorem{theorem}{Theorem}
\newtheorem{proposition}[theorem]{Proposition}

\newtheorem{lemma}[theorem]{Lemma}

\newtheorem{corollary}{Corollary}
\newtheorem{definition}[theorem]{Definition}

\newtheorem{claim}{Claim}

\def\Z{\mathbb{Z}}

\def\C{\mathbb{C}}

\def\Q{\mathbb{Q}}

\def\R{\mathbb{R}}
\def\C{\mathbb{C}}

\def\N{\mathbb{N}}

\def\F{\mathbb{F}}

\def\md{\mathcal{D}}

\def\Zpk{\mathbb{Z}/p^{k}}
\def\Zpk1{\mathbb{Z}/p^{k-1}}

\newcommand{\rref}[1]{(\ref{#1})}

\newcommand{\beg}[2]{\begin{equation}\label{#1}#2\end{equation}}
\def\r{\rightarrow}

\def\F{\mathbb{F}}

\def\sl2{\widetilde{SL_{2}(\Z)}}

\def\md
\def\rank{\operatorname{rank}}

\title[]{Howe duality over finite fields I: The two stable ranges}

\author{Sophie Kriz}
\thanks{The author was supported by a 2023 National Science Foundation
Graduate Research Fellowship, no. 2023350430}

\begin{document}

\maketitle
\vspace{-5mm}
\begin{abstract}

This is the first in a series of papers on type I Howe duality for finite fields, concerning the restriction of an oscillator representation of the symplectic group to a product of a symplectic and an
orthogonal group. The goal of the series is describing this restriction completely
explicitly.
Applications (described in the third paper of the series) include demonstrating that
the tensor pairs previously calculated by S.-Y. Pan as occuring with non-zero multiplicity
occur with multiplicity 1, proving the type C case of the Gurevich-Howe rank conjecture,
and giving a recursive formula for the characters of cuspidal unipotent representations.
In this first paper,
we construct the correspondence in the two so called stable ranges, where the rank of one of the factors is large enough with respect to the other.

\end{abstract}

\vspace{-3mm}

\tableofcontents

\vspace{-10mm}

\section{Introduction}

This is the first paper in a series dedicated to the problem of Howe correspondence for finite fields.
The question was first introduced by R. Howe in \cite{HoweFiniteFields} 
and was discussed in the papers \cite{HoweGurevich, HoweGurevichBook}
by S. Gurevich and Howe, bringing
to light the finite field case of the Howe duality conjecture \cite{HoweTheta} for locally compact fields.
In rough terms, the correspondence concerns the decomposition of the
restriction of the oscillator representation of a finite symplectic group to a product of a symplectic and an orthogonal group which are each other's centralizers. 
More broadly, this can be discussed for general reductive dual pairs, consisting of any pair of subgroups
which are each other's centralizers in a finite symplectic group.

Early progress on this program for the case of finite fields 
was made by J. Adams and A. Moy \cite{AdamsMoy},
who described the behavior of unipotent cuspidal representations under the Howe correspondence
between symplectic and even orthogonal groups,
using previous results of S. Kudla \cite{Kudla} for the theta correspondence on local fields,
which used the Jacquet functor. Another important
advance was the paper \cite{AubertMichelRouquier} by A.-M. Aubert, J. Michel, and R. Rouquier,
who conjectured the behavior of the Howe correspondence on general
{\em unipotent representations} (i.e. those appearing in the Deligne-Lusztig induction of
a trivial character), and proved their conjecture for unitary and general linear groups
(i.e. the ``type II case" of Howe duality). Their conjecture for {\em type I dual pairs},
i.e. symplectic and orthogonal groups, was proved by S.-Y. Pan \cite{Pan1} using the
concept of a {\em uniform projection} (a character is called uniform if it is a linear
combination of virtual Deligne-Lusztig characters).
This geometric approach was also used 
by D. Liu and Z. Wang in \cite{LiuWang} to extend the results of Adams and Moy
to the case of reductive dual pairs involving odd orthogonal groups.
Pan \cite{Pan2} went
further and, by proving suitable compatibility results, was able to determine which pairs of irreducible
representations occur in the type I Howe duality correspondence with non-zero multiplicity.

In the present series of papers, we approach this problem using
a different method. We study the Howe duality
correspondence directly and explicitly by examining the endomorphism algebras of
tensor products of oscillator representations over finite fields. In these terms, we describe
explicitly the so-called {\em $\eta$-correspondence} defined by Gurevich and Howe \cite{HoweGurevich, HoweGurevichBook}, and also define and similarly describe a complementary
``$\zeta$-correspondence." Having a hands-on description of the irreducible representations
occuring in the Howe duality correspondence gives us more precise information. For example,
as a payoff, we are able to explicitly demonstrate that the pairs identified by Pan \cite{Pan2}
always occur with multiplicity $1$. Using the organization
of the oscillator representation's decomposition
in terms of the eta and zeta correspondences, we obtain a recursive formula for the characters
of the unipotent cuspidal representations, which were still quite mysterious up to this point.
We also prove the type C case of the Gurevich-Howe rank conjecture 
(Conjecture 0.3.8 of \cite{HoweGurevichBook}),
the type A case of which was proved R. M. Guralnick, M. Larsen, and P. H. Tiep \cite{PiGuralLarsTiep},
and the type B and D cases of which were proved by Larsen and Tiep in \cite{LarsenTiep}.
These applications will be discussed in the third paper of the series.



The present first paper of the series constructs the correspondence in two so called stable ranges, in which all irreducible representations of either the symplectic or the orthogonal group occur. We precisely construct the summands. (In the symplectic stable range, the summands were previously described by
F. Montealegre-Mora and D. Gross \cite{QuantumCodes}
in a less explicit way.) In the next paper, we will make these
summand explicit using the Lusztig classification of representations of
finite groups of Lie type, (extending the results of S.-Y. Pan
\cite{Pan1, Pan2}). In the last paper of the series,
we extend that descriptions beyond the stable ranges, to cover all cases of the Howe correspondence.

\vspace{3mm}

For a symplectic vector space $\mathbf{V}$ over a field $k$, given a non-trivial additive character
$$k \rightarrow \C ,$$
one may consider the associated {\em oscillator} or {\em Weil-Shale} representation $\omega$
of the symplectic group $\text{Sp}(\mathbf{V})$ (when, for example, $k = \C$ or $\F_q$), or the
metaplectic group $Mp (\mathbf{V})$ (when $k = \R$ or $\Q_p$). 
The question of Howe duality asks about the decomposition of the restriction
of $\omega$ to the product of a reductive dual pair of subgroups of $\text{Sp} (\mathbf{V})$
or $Mp (\mathbf{V})$.
Denote this reductive dual pair by $G, H$. Consider the decomposition
\beg{GenRedDualPairRes}{\text{Res}_{G \times H} (\omega) = \bigoplus_{\rho \in \widehat{G} , \pi \in \widehat{H}}
\mu( \rho, \pi) \cdot (\rho \otimes \pi) }
for some multiplicities
$$\mu : \widehat{G} \times \widehat{H} \rightarrow \N_0.$$
(In this paper, for a group $G$, $\widehat{G}$ denotes
the set of irreducible unitary $G$-representations.)

For fields $k = \C, \R,$ or $\Q_p$, there in fact exist subcollections
$\mathcal{S}_{G}$ and $\mathcal{S}_{H}$ of the irreducible representations
of $G$ and $H$, and
a bijective correspondence (called the {\em theta correspondence})
$$\theta : \mathcal{S}_{G} \rightarrow \mathcal{S}_{H},$$
such that
\beg{ThetaCorrStat}{\begin{array}{c}
\text{Res}_{G\times H} (\omega) =\\[1ex]
\displaystyle
 \int_{\rho \in \mathcal{S}_{G}} \rho \otimes \theta (\rho) = \int_{\pi \in \mathcal{S}_{H}} \theta^{-1} (\pi ) \otimes \pi.
\end{array}}
This result, referred to as {\em Howe duality},
is the culmination of a long history of work, (see, for example,
R. Howe, \cite{HoweTheta} for the case of $k = \R$ and W.T. Gan, S. Takeda, \cite{GanTakeda}
for $k = \Q_p$), and has many deep applications
in arithmetic geometry, number theory, and representation theory.
It is an interesting and difficult question to
find its appropriate analogue in the case when $k$ is a finite field.
This case, where $k = \F_q$ (for $q$ a power of an odd prime),
is the main topic of the present note.

\vspace{5mm}

Due to the large isotropic subspaces of symmetric bilinear forms over
a finite field, a bijective correspondence $\theta$ giving a decomposition \rref{ThetaCorrStat} is
not the correct formalism. Still, one can study the decomposition of
$\text{Res}_{\text{Sp} (V) \times \text{O}(W, B)} (\omega [V\otimes W])$, which exhibits interesting
relationships between irreducible representations of $\text{Sp}(V)$ and $\text{O}(W,B)$,
and many meaningful patterns regarding which pairs $\rho \otimes \pi$
appear in the decomposition \rref{GenRedDualPairRes} have been 
constructed and studied, see for example
\cite{AubertKP, AubertMichelRouquier, Chavez, HoweGurevich, HoweGurevichBook, Pan1, Pan2}.

\vspace{5mm}

As in the context of an infinite locally compact field,
over a finite field, we may still choose to either
consider the decomposition of the restricted oscillator representation to be indexed 
by the irreducible representations of $\text{Sp}(V)$
\beg{Spenumeration}{\text{Res}_{\text{Sp}(V) \times \text{O}(W,B)} (\omega[W \otimes V]) = \bigoplus_{\rho \in \widehat{\text{Sp} (V)}} \rho \otimes \Psi (\rho) }
or by the irreducible representations
of $\text{O}(W)$
\beg{Oenumeration}{\text{Res}_{\text{Sp}(V) \times \text{O}(W,B)}(\omega [ W \otimes V] ) = \bigoplus_{\rho \in \widehat{\text{O}(W)}}  \Phi(\rho)\otimes \rho,}
for some (not necessarily irreducible or non-zero) $\text{Sp}(V)$- and $\text{O}(W,B)$-representations
$\Phi (\rho )$ and $\Psi (\rho)$.
We investigate \rref{Spenumeration} or \rref{Oenumeration} by examining the
$\text{Sp}(V)$- or $\text{O}(W,B)$-equivariant endomorphism algebras
$$\text{End}_{\text{Sp}(V)} (\omega [V\otimes W]), \hspace{5mm} \text{End}_{\text{O}(W,B)} (\omega [V\otimes W]),$$
respectively (note that in this notation, the choice of $B$ is assumed
to specify the restriction of $\omega [V\otimes W]$ to $\text{Sp}(V)$,
but is hidden in the notation).
It intuitively makes sense to consider \rref{Spenumeration} (resp. \rref{Oenumeration})
if in the reductive dual
pair $(\text{Sp}(V), \text{O}(W,B))$ the symplectic group is ``much larger" (resp. ``much smaller")
than the orthogonal group.

\vspace{5mm}

The case of reductive dual pairs $(\text{Sp}(V), \text{O}(W,B))$
where $V$ has a large enough dimension compared to $W$ (specifically, where $dim (V) \geq 2 dim (W)$),
which we will refer to as the {\em symplectic stable case}, has been especially studied.
S. Gurevich
and R. Howe \cite{HoweGurevich} proved that in this case, every irreducible representation $\rho$ of $\text{O}(W,B)$
appears in the restriction of $\text{Res}_{\text{Sp}(V) \times \text{O}(W,B)} (\omega [ V\otimes W])$ further to $\text{O}(W,B)$,
and therefore, there exists a non-zero representation of $\text{Sp}(V)$ whose tensor product with $\rho$ is a
summand of $\text{Res}_{\text{Sp}(V) \times \text{O}(W,B)} (\omega [ V\otimes W])$. Further, there is
a unique ``top" irreducible piece of this representation giving the {\em eta correspondence}
$\eta (\rho)$,
which was constructed by S. Gurevich and R. Howe to specifically have highest {\em $U$-rank}
(see \cite{HoweGurevich, HoweGurevichBook} for more details).
Here, we write
$$\eta^V_{W,B} : \widehat{\text{O}(W,B)} \hookrightarrow \widehat{\text{Sp}(V)}$$
for the eta correspondence
(we omit the subscript and superscript when they are clear).

\vspace{5mm}

Examining the structure of the endomorphism algebra of oscillator representations
(and their tensor powers) can be used to investigate $\eta^V_{W,B}$
and derive an explicit decomposition of the restricted oscillator representation in terms of it.
In fact, from this perspective, the roles of the symplectic and orthogonal groups can be switched.
In this paper, we find that for an opposite {\em orthogonal stable range} (where
$dim (V)$ is less than or equal to the dimension of the maximal isotropic subspace
of $W$), there exists a system of injections
$$\zeta^{W,B}_V: \widehat{\text{Sp}(V)} \hookrightarrow \widehat{\text{O}(W,B)}  $$
we call the {\em zeta correspondence}, along with an accompanying explicit decomposition
of the restricted oscillator representations. Like $\eta^V_{W,B}$ in the symplectic
stable range, $\zeta^{W,B}_V$ all have disjoint images in the orthogonal stable range,
each giving a new range of $O(W,B)$-representations.
To state this concretely, we introduce some notation.

\vspace{3mm}

Consider an $n$-dimensional orthogonal space and form
bilinear form $(W,B)$, and a $2N$-dimensional symplectic space and form $(V,S)$. 
The maximal dimension of an isotropic subspace of $V$ with respect to $S$ is $N$. 
We denote by $h_W$ the maximum dimension of
an isotropic subspace of $W$ with respect to $B$. Over a finite field, in the case
when $n= 2m $ is even, we either have
$h_W = m-1$ or $m$
(giving rise to the two orthogonal groups $O^+_{2m} (\F_q)$ and $O^-_{2m} (\F_q)$, respectively).
In the case when $n= 2m+1$ is odd, then we must have $h_W= m$.
Then $B$ can be expressed as a direct sum of $h_W$ copies of the 2-dimensional
hyperbolic symmetric bilinear form $\begin{pmatrix} 0 & 1\\ 1& 0 \end{pmatrix}$,
and a final anisotropic part of dimension $0$, $1$, or $2$.
For $k$ the dimension of an isotropic space $Z$ in $V$ (resp. $W$), we denote by
$(V[-k], S[-k])$ (resp. $(W[-k], B[-k])$) the subspace of dimension
$2N-2k$ (resp. $n-2k$) and its accompanying
non-degenerate symplectic (resp. symmetric) form, which are obtained by projecting away from
$Z$ and its dual in $V$ with respect to $S$ (resp. $B$).
Let us also denote by $P^V_{ k}$ (resp. $P^B_{ k}$) the parabolic subgroups
of $\text{Sp} (V)$ (resp. $\text{O}(W, B)$) corresponding to an $k$-dimensional
isotropic subspace of $V$ (resp. $W$) with Levi subgroup
$$\text{GL}_{k} (\F_q) \times \text{Sp}_{2N-2k} (\F_q)$$
(resp.
$$\text{GL}_k (\F_q) \times \text{O}_{n-2k} (\F_q)).$$
We shall consider the parabolic induction of a representation of
the Levi subgroup by inflating it to a representation of the parabolic by
taking unipotent elements to act trivially, and then inducing to a representation of
the full symplectic or orthogonal group. 

\vspace{2mm}

\noindent {\bf Convention:} 
In this series of papers, we shall adhere to the standard convention in this topic, which is, for a subgroup
$H \subseteq G$, to write $\text{Ind}^G_H$ for induction from $H$ to $G$ and $\text{Res}^G_H$ for restriction from $G$ to $H$.
If one of the variables is understood from the context, we sometimes omit it.

\vspace{2mm}

The stable ranges are characterized by the smaller of the symplectic and orthogonal spaces
being of a lesser or equal dimension to the maximal dimension of an isotropic subspace of the other
space and form:
\begin{definition}
Consider a reductive dual pair $(\text{Sp}(V), \text{O}(W,B))$ of subgroups of $\text{Sp}(V \otimes W)$. Assume the previous paragraph's notation. 
\begin{enumerate}
\item We say $(\text{Sp}(V), \text{O}(W,B))$ is in the {\em symplectic stable range} if
$n \leq N$.

\vspace{3mm}

\item We say $(\text{Sp}(V), \text{O}(W,B))$ is in the {\em orthogonal stable range} if
$2N \leq h_W$.

\end{enumerate}
\end{definition}

\vspace{3mm}

\noindent The purpose of this paper is then to prove the following

\begin{theorem}\label{TheoremGenuine}
There are mirroring results in the two stable ranges of reductive dual pairs:
\begin{enumerate}
\item \label{SymTheoremGenuine} For $(\text{Sp}(V), \text{O}(W,B))$ in the symplectic stable range
(which we recall means that $\text{dim} (W) \leq \text{dim} (V)/2$), we find
\beg{SympStableEndDecompThm}{\text{End}_{\text{Sp}(V)} (\omega [ V\otimes W]) = \bigoplus_{k=0}^{h_W} M_{|\text{O}(W,B)/P_k^B|}
(\C \text{O} (W[-k], B[-k]))}
giving an explicit isomorphism of $\text{Res}_{\text{Sp}(V) \times \text{O}(W,B)} (\omega[ V\otimes W])$ with 
\beg{EtaCorrThmDecomp}{
\bigoplus_{k=0}^{h_W}
\bigoplus_{\rho \in \widehat{\text{O}(W[-k], B[-k])}} \eta^V (\rho) \otimes \text{Ind}_{P_k^B} (\rho \otimes \epsilon (\text{det}))
}
where $\epsilon (\text{det})$ is considered as a representation of the $\text{GL}_k (\F_q)$ factor
of the Levi factor of $P_k^B$ in each term.

\vspace{5mm}

\item \label{OrthoTheoremGenuine} 
For $(\text{Sp}(V), \text{O}(W,B))$ in the orthogonal stable range
(which we recall means that $\text{dim} (V)$ is less than or equal to the
maximal dimension of a $B$-isotropic subspace of $W$), we find
\beg{OrthoStableEndDecompThm}{\text{End}_{\text{O}(W,B)} (\omega [ V\otimes W]) \cong \bigoplus_{k=0}^N M_{|\text{Sp}(V)/ P_k^V|} (\C \text{Sp}(V [-k]))}
In fact, there exists a system of injections
$$\zeta^{(W,B)}_V: \widehat{\text{Sp}(V)} \hookrightarrow \widehat{\text{O}(W,B)}$$
with disjoint images such that $\text{Res}_{\text{Sp} (V) \times \text{O}(W,B)} (\omega [ V\otimes W])$ decomposes as
\beg{ZetaCorrThmDecomp}{ \bigoplus_{k=0}^N
\bigoplus_{\rho \in \widehat{\text{Sp}(V[-k])}} \text{Ind}_{P_k^V} (\rho \otimes \epsilon (\text{det}))
\otimes \zeta^{(W,B)} (\rho),}
where $\epsilon (\text{det})$ is considered as a representation of the $\text{GL}_k (\F_q)$ factor
of the Levi factor of $P_k^V$ in each term.

\end{enumerate}
\end{theorem}
\vspace{3mm}

\noindent{\bf Comment:}
The existence of the isomorphism \rref{EtaCorrThmDecomp} was established
from the point of
view of mathematical physics by F. Montealegre-Mora
and D. Gross \cite{QuantumCodes}.
In this paper, we follow a different approach
which can be applied to both the symplectic and orthogonal stable ranges, and
which gives
an explicit combinatorial formula for the decomposition.

\vspace{3mm}

In the stable ranges, we can use these results to find an explicit description of $\eta^V_{W,B}$
and $\zeta^{W,B}_V$ in terms of Lusztig's classification of irreducible representations of finite groups
of Lie type, using information about the possible dimensions of irreducible representations of $\text{Sp}(V)$,
$O(W,B)$. This is done in \cite{KrizLusztigStablSymp},
(extending the results of
\cite{Pan1, Pan2}).

These results obtain explicit decompositions of $\text{Res}_{\text{Sp}(V)\times \text{O}(W,B)} (\omega [ V\otimes W])$
in the symplectic and orthogonal stable ranges, covering about half of the general cases of
type I reductive dual pairs $(\text{Sp}(V), \text{O}(W,B))$. One may ask if Theorem
\ref{TheoremGenuine} and the following results in \cite{KrizLusztigStablSymp} can be extended into the unstable
ranges, and specifically the cases up to and including the ``middle" where the ranks of $\text{Sp}(V)$ and $\text{O}(W,B)$ are equal or close. This is possible using {\em interpolated category theory},
and this result is concluded in \cite{KrizLusztigExtended}
(completing a full description of $\text{Res}_{\text{Sp}(V) \times \text{O}(W,B)} (\omega[V\otimes W])$ for a general case of $(\text{Sp}(V), \text{O}(W,B))$).

\vspace{5mm}

The present paper is organized as follows: In Section \ref{BasicEndSection},
we discuss the endomorphism algebras of
$\omega[ V\otimes W]$ over $\text{Sp}(V)$ and $\text{O}(W,B)$.
We interpret endomorphisms fixed points in $\C V \otimes W$ and introduce an operation
$\star$ corresponding to $\circ$. Using this description, we can elementarily calculate
the dimensions
of $\text{End}_{\text{Sp}(V)} (\omega [V\otimes W])$ and $\text{End}_{\text{O}(W,B)} (\omega [ V\otimes W])$
and find elements 
corresponding to generators of the group algebras of orthogonal and symplectic groups,
respectively.
In Section \ref{CombinSection}, we prove a combinatorial identity verifying that the dimensions of both sides
of \rref{SympStableEndDecompThm}, \rref{OrthoStableEndDecompThm} match,
and therefore they are equal.
In Section \ref{ProofSection}, we use an inductive argument and information about $Hom$-spaces
between oscillator representations (and their tensor powers) to conclude Theorem \ref{TheoremGenuine}.

\vspace{5mm}

\noindent {\bf Acknowledgment:} The author is thankful to Pierre Deligne, Shamgar Gurevich, Roger Howe, 
Peter Sarnak, Pham Huu Tiep, and
Jialiang Zou for discussions and comments.

\vspace{5mm}

\section{The endomorphism algebra of an oscillator representation}\label{BasicEndSection}

\vspace{3mm}

In this section, we discuss the structure of the endomorphism algebra of a tensor products
of oscillator representations. In Subsection \ref{GeneralEndOscStructure}, we
express the endomorphism algebra over $G \subseteq \mathbf{V}$
of an oscillator representation $\omega [ \mathbf{V}]$
as isomorphic to the fixed point space $\C \mathbf{V}^G$ with a certain algebra operation.
In Subsection \ref{CountingOrbitsSubSect}, we use this to compute the dimension
of the endomorphism algebra of $\omega [ V\otimes W]$ over $\text{Sp}(V)$ and $\text{O}(W,B)$.
In Subsection \ref{TopSubSectSymp}, we discuss the ``top part" of the
endomorphism algebra of a restricted oscillator representation over $\text{Sp}(V)$ or $\text{O}(W,B)$,
and find the group algebra 
$\C \text{O}(W,B)$ in the top part over $\text{Sp}(V)$, in the case of $(\text{Sp}(V), \text{O}(W,B))$ in the symplectic stable range.
In Subsection \ref{TopSubSectOrtho}, we find the group algebra $\C \text{Sp}(V)$ in the top part of the endomorphism
algebra of $\omega[ V\otimes W]$ over $\text{O}(W,B)$.

\subsection{Endomorphisms and the Schr\"{o}dinger model}\label{GeneralEndOscStructure}

First, consider a general symplectic space and form $(\mathbf{V}, \mathbf{S})$.
Recall that for an oscillator representation $\omega_a [ \mathbf{V}]$, its dual
is the oscillator representation of opposite character $\omega_{-a} [ \mathbf{V}]$
(see \cite{HoweFiniteFields}). Their tensor product gives the standard representation
$\C \mathbf{V}$, where $\text{Sp} (\mathbf{V})$ acts geometrically.

Hence, for a subgroup $G \subseteq \text{Sp} (\mathbf{V})$,
the endomorphism algebra of the restriction of the
oscillator representation to $G$ is, as a vector space, 
$$\text{End}_G (\omega_a [\mathbf{V}]) \cong \text{Hom}_G (1, \omega_a [\mathbf{V}] \otimes \omega_{-a} [\mathbf{V}])
\cong (\C \mathbf{V})^G,$$
which can also be considered as the $\C$-vector space with a
basis indexed by $G$-orbits on $\mathbf{V}$.

In fact, we can
introduce an operation $\star_{\mathbf{V}}$ on $\C \mathbf{V}$ corresponding to composition
in the endomorphism algebra such that, as $\C$-algebras,
$$(\text{End}_{\text{Vect}} (\omega_a [\mathbf{V}] ), \circ) \cong (\C (\mathbf{V}), \star_{\mathbf{V}})$$
(for subgroups $G\subseteq \text{Sp} (\mathbf{V})$, the endomorphism algebra of $\omega [ \mathbf{V}]$
over $G$ is again isomorphic to the subalgebra generated by $G$-orbits on $\mathbf{V}$).
We put, for $u, v\in \mathbf{V}$,
$$ (u) \star_{\mathbf{V}} (v) = \psi_a (\frac{1}{2} \mathbf{S}(u,v)) \cdot (u + v).$$
This can be considered as an ``untwisted" variant
of the algebra operation arising from the {\em Schr\"{o}dinger model} of
the oscillator representation:

Recall that for a decomposition of a symplectic space
$\mathbf{V}$ into Lagrangians $\boldsymbol \Lambda_+ \oplus \boldsymbol \Lambda_-$,
we may identify $\omega_a [ \mathbf{V}] \cong \C \boldsymbol \Lambda_-$. The action of
an element $(v, c)$ of the Heisenberg group $\mathbb{H}$ (for $v\in \mathbf{V}$, $c\in \F_q$)
on a generator $x \in \Lambda_-$ is given by
$$(v,c) (x)  = \psi_a (c + \frac{1}{2} \mathbf{S}(v_+, x)) \cdot (v_-+x),$$
where $v= v_+ + v_-$ is the unique decomposition of a vector $v$ into $v_+ \in \boldsymbol \Lambda_+$,
$v_- \in \boldsymbol \Lambda_-$. This gives $\omega_a [\mathbf{V}]$ the structure
of a {\em Weil-Shale representation}. The action of $\text{Sp}(\mathbf{V})$ making it
an oscillator representation arises from the uniquess of the Weil-Shale representation
for each central character. We then see a natural action of an algebra $(\C \mathbf{V}, \ast_{\mathbf{V}})$
for algebra operation $\ast$ given by, for $u = u_+ + u_-, v= v_+ + v_- \in \mathbf{V}$,
$u_\pm, v_\pm \in \boldsymbol \Lambda_\pm$,
$$(u) \ast (v) = \psi_a (\mathbf{S}(u_+, v_-)) \cdot (u+v)$$
(See also Remark 7.2.6 of \cite{HoweGurevich}, identifying a vector generator $(v)$ with its indicator function.)

Applied to $x \in \boldsymbol \Lambda_-$, a vector $v \in (\C \mathbf{V}, \ast_{\mathbf{V}})$ acts by
$$(v) (x) = \psi_a (\mathbf{S}(v_+, x)) \cdot (v_- + x).$$

Now our choice of algebra $(\C \mathbf{V}, \star_{\mathbf{V}})$ is isomorphic to the algebra
$(\C \mathbf{V}, \ast_{\mathbf{V}})$ along
$$
(\C \mathbf{V}, \star_{\mathbf{V}}) \r (\C \mathbf{V}, \ast_{\mathbf{V}})$$
$$(v) \mapsto \psi_{a} (\frac{1}{2} \mathbf{S}(v_+, v_-)) \cdot (v)$$
for $v = v_+ + v_-\in \mathbf{V}$ with $v_\pm \in \boldsymbol \Lambda_\pm$.
Therefore, an element $v \in ( \C \mathbf{V}, \star)$ acts on $x \in \boldsymbol \Lambda_-$
by
\beg{TrueStarAction}{(v) (x) = \psi_a (\mathbf{S} (v_+ , x) + \frac{1}{2} \mathbf{S} (v_+, v_-))
(v_- + x)}
for $v = v_+ + v_-$ with $v_\pm \in \boldsymbol \Lambda_\pm$.

For a vector $v \in \mathbf{V}$, considering it as an endomorphism
$(v) \in \C \mathbf{V} = \text{End}_{\text{Vect}} (\omega_a [\mathbf{V}])$, its
trace is
$$tr (v) = \begin{cases}
0, \text{ for } v \neq 0\\
|\mathbf{V}| \text{ for } v=0
\end{cases}
$$
where $|\mathbf{V}|$ denotes the set order of $\mathbf{V}$.

\vspace{5mm}

Applying this to the present situation, where $\mathbf{V}= W \otimes V$, $\mathbf{S} = B \otimes S$,
we have
$(\text{End}_{\text{Vect}} (\omega [W \otimes V]), \circ) \cong (\C (W \otimes V), \star_{W \otimes V}$. 
For $u_1, u_2 \in W$, $v_1, v_2 \in V$
\beg{FullGeneralityStarVOW}{\begin{array}{c}
(v_1 \otimes w_1) \star_{V \otimes W} (v_2 \otimes w_2) = \\[1ex]
\displaystyle \psi ( \frac{S (v_1 , v_2) \cdot B (w_1, w_2)}{2})
\cdot (v_1 \otimes w_1 + v_2 \otimes w_2).
\end{array}}
When the ground space is clear, we omit the subscript in $\star$ (in this context,
it will typically be $V \otimes W$). In particular, we have that the $\text{Sp}(V)$- and $\text{O}(W,B)$-equivariant
endomorphism algebras on $\omega [V\otimes W]$ can be expressed as
\beg{}{(\text{End}_{\text{Sp}(V)} (\omega[V\otimes W]), \circ ) \cong (\C (V\otimes W)^{\text{Sp}(V)}, \star)}
\beg{}{(\text{End}_{\text{O}(W,B)} (\omega[V\otimes W]), \circ ) \cong (\C( V\otimes W)^{\text{O}(W,B)}, \star)}

\vspace{5mm}

\subsection{Counting orbits}\label{CountingOrbitsSubSect}

\vspace{5mm}

In this subsection, we calculate the dimensions
$$\begin{array}{c}
\text{dim} (\text{End}_{\text{Sp} (V)} (\omega [ V\otimes W])) = \text{dim} (\C (V \otimes W)^{\text{Sp}(V)}),\\
\text{dim} (\text{End}_{\text{O}(W,B)} (\omega[ V\otimes W]))= \text{dim} (\C (V \otimes W)^{\text{O}(W,B)})
\end{array}$$
by counting the numer of orbits of $\text{Sp}(V)$ on $V\otimes W= V^{\oplus n}$ and the number
of orbits of $\text{O}(W,B)$ on $V\otimes W = W^{\oplus 2N}$, respectively.

\begin{lemma}\label{FixedPointCountingLemma}
Consider a symplectic space $V$ of dimension $2N$ and an space with symmetric bilinear
form $(W,B)$ of dimension $n$ with maximal dimension of an isotropic subspace denote by $h_W$.
\begin{enumerate}

\item \label{FixedPointLemmaSpCase} If $(\text{Sp}(V), \text{O}(W,B))$ for a reductive dual
pair in the symplectic stable range i.e. $n \leq N$, then
\beg{SpDiagonalFixedPointCount}{\text{dim} (\C (V \otimes W)^{\text{Sp}(V)})) = 2 (q+1)\dots (q^{n-1} +1)}

\vspace{3mm}

\item \label{FixedPointLemmaOrthoCase}
If $(\text{Sp}(V), \text{O}(W,B))$ for a reductive dual pair in the orthogonal stable range
i.e. $2N \leq h_W$, then
\beg{SpDiagonalFixedPointCount}{\text{dim} (\C (V \otimes W)^{\text{O}(W,B)}) = (q+1) (q^2+1)\dots (q^{2N} +1)}

\end{enumerate}
\end{lemma}

\begin{proof}
We begin with the proof of \rref{FixedPointLemmaSpCase}. We write $S$ for
the symplectic form on $V$. For
\rref{SpDiagonalFixedPointCount}, we need to compute the number of $\text{Sp}(V)$-orbits
on $V\otimes W = V^{\oplus n}$ (with $\text{Sp}(V)$ acting diagonally).
First recall that orbits of this $\text{Sp}(V)$ action on $n$-tuples of vectors in $V$ correspond to the data of
\begin{itemize}

\item A $d \times n$ matrix $M$ over $\F_q$ in reduced row echelon form for some $0 \leq d \leq n$.

\item A choice of $d\choose 2$ scalars $a_{i < j}$ corresponding to each pair of indices
$i< j \in \{1 , \dots , d\}$.

\end{itemize}
Given this data, its corresponding $\text{Sp}(V)$-orbit is
\beg{SpOrbitCorrRREFScalars}{\begin{array}{r}
\{(v_1, \dots , v_d) \cdot M\mid v_1, \dots , v_d \in V \text{ linearly independent},\\
\text{for every } i< j \hspace{5mm} S(v_i, v_j) = a_{i< j}\}
\end{array}}
(where $(v_1,\dots , v_d) \cdot M$ is computed by treating $(v_1, \dots , v_d)$ as a $1 \times d$ matrix,
so that the result is a $1 \times n$ matrix of vectors in $V$, i.e. is an element of $V^{\oplus n}$).
Now the range condition that $n \leq N$ precisely ensures that every choice of data
$(M, a_{i< j})$ gives a non-empty orbit \rref{SpOrbitCorrRREFScalars}, since
\rref{SpOrbitCorrRREFScalars} is non-empty if and only if
$$dim (V) \geq d+ dim (ker (S|_{\langle v_1, \dots , v_d\rangle}))$$
where $S|_{\langle v_1, \dots , v_d\rangle}$ denotes the possibly degenerate symplectic form
obtained from restricting $S$ to the span of $v_1, \dots , v_d$.

Therefore, it suffices to count all choices of data $(M, a_{i< j})$. Given $M$, the number of choices
of scalars $(a_{i<j})$ is $q^{d\choose 2}$. Now the number of $d\times n$ matrices in reduced
row echelon form is well-known to be the $q$-combination number ${n \choose d}_q$, giving
\beg{HalfwaydownSpOrbitCount}{\text{dim} (\text{End}_{\text{Sp}(V)} (\omega [ V\otimes W])) = \sum_{d=0}^n {n \choose d}_q \cdot q^{d \choose 2}.}
To see the simplification to \rref{SpDiagonalFixedPointCount} elementarily, it is actually easier to naively
write down the number of reduced row echelon $d \times n$ matrices as
$${n \choose d}_q = \sum_{1 \leq \ell_1 < \dots < \ell_d \leq n} q^{dn - (\ell_1 + \dots + \ell_d)- {d \choose 2}},$$
taking $\ell_i$ to be the length of each row from the left up to an including the pivot, which leaves 
exactly $dn - (\ell_1+ \dots + \ell_d) - {d \choose 2}$ un-determined entries (the 
${d\choose 2} = 1+ 2+ \dots + (d-1)$ term arises from the entries directly above a pivot
being $0$). Using this in \rref{HalfwaydownSpOrbitCount} gives
$$\begin{array}{c}
\displaystyle \text{dim} (\text{End}_{\text{Sp}(V)} (\omega [ V\otimes W])) = \sum_{1 \leq \ell_1< \dots < \ell_d \leq n}
q^{dn - (\ell_1+ \dots + \ell_d)} =\\
\displaystyle \sum_{1\leq \ell_1< \dots < \ell_d \leq n} \prod q^{n-\ell_i} = \prod_{j=0}^{n-1} (q^j+1),
\end{array}$$
as claimed.

\vspace{3mm}

The proof of \rref{FixedPointLemmaOrthoCase} follows similarly:
As in the symplectic case, the orbits of $O(W,B)$ on $V \otimes W = W^{\oplus 2N}$ correspond
to the data of 
\begin{itemize}

\item A $d \times 2N$ matrix $M$ over $\F_q$ in reduced row echelon form for some $0 \leq d \leq 2N$.

\item A choice of ${d+1 \choose 2}= {d\choose 2} + d$
scalars $a_{i \leq j}$ corresponding to each pair of indices
$i\leq  j \in \{1 , \dots , d\}$.

\end{itemize}
The orbit corresponding to $(M, a_{i\leq j})$ is
\beg{OrthoOrbitCorrRREFScalars}{\begin{array}{r}
\{(v_1, \dots , v_d) \cdot M\mid v_1, \dots , v_d \in V \text{ linearly independent},\\
\text{for every } i\leq j \hspace{5mm} B(v_i, v_j) = a_{i\leq  j}\}
\end{array}}
(again, treating $(v_1, \dots , v_d)$ as a $1 \times d$ matrix). As before, the range condition
precisely ensures that each \rref{OrthoOrbitCorrRREFScalars} from any choice
of $(M, a_{i \leq j})$ is non-empty. (We note that the cardinality of \rref{OrthoOrbitCorrRREFScalars},
like the range condition, depends on the form of $B$ when $n$ is even.)
Therefore, by the above argument, we have
$$\text{dim} (\text{End}_{\text{O}(W,B)} (\omega [V\otimes W])) = \sum_{d=0}^{2N} {2N \choose d}_q \cdot
q^{{d\choose 2} +d },$$
which can be simplified using the naive expression for ${2N \choose d}_q$ to
$$\begin{array}{c}
\displaystyle \text{dim} (\text{End}_{\text{O}(W,B)} (\omega [ V\otimes W])) = \sum_{1 \leq \ell_1< \dots < \ell_d \leq 2N}
q^{d(2N+1) - (\ell_1+ \dots + \ell_d)}=\\
\displaystyle \sum_{1 \leq \ell_1 < \dots < \ell_d \leq 2N} \prod q^{2N+1 - \ell_i} = 
\prod_{j=1}^{2N} (q^j+1),
\end{array}$$
as claimed.

\end{proof}

\vspace{3mm}

\subsection{The ``top subalgebra" and generators of the orthogonal group}\label{TopSubSectSymp}

First note that
\beg{ResSpAsTens}{\text{Res}_{\text{Sp} (V)} (\omega [ V\otimes W]) \cong
\omega_{a_1} [ V] \otimes \dots \otimes \omega_{a_n} [V]}
where $a_1, \dots , a_n$ are the eigenvalues of $B$ and
\beg{ResOrthoAsTens}{\text{Res}_{\text{O}(W,B)} (\omega [ V\otimes W]) =
(\text{Res}_{\text{O}(W,B)} (\omega [ \F_q^2 \otimes W])^{\otimes N},}
where we consider $\omega [ \mathbb{F}_q^2 \otimes W]$ to be the oscillator representation
on $\mathbb{F}_q^2$ with the standard symlectic form 
tensored with $(W,B)$. We take its restriction along the inclusion
$$\text{O}(W, B) \hookrightarrow \text{SL}_2 (\F_q) \times O(W,B)$$
(writing $\text{SL}_2 (\F_q) = \text{Sp} (\F_q^2)$). In fact, we note that actually the restriction
\beg{PermOWB}{\text{Res}_{\text{O}(W,B)} (\omega [\F_q^2 \otimes W]) \cong \C W^-}
is isomorphic to the (twisted) permutation representation of $O(W,B)$, where $-$ indicates
the sign character of $\Z/2 = O(W,B)/SO(W,B)$,
(by considering first the restriction to $GL(W)$).

Now to approach both the symplectic and orthogonal stable ranges, our strategy will be to first consider the {\em top subalgebras}
$$\text{End}_{\text{Sp}(V)}^{top} (\omega [ V\otimes W]) \subseteq \text{End}_{\text{Sp}(V)} (\omega [ V\otimes W]),$$
$$\text{End}_{\text{O}(W,B)}^{top} (\omega [ V\otimes W]) \subseteq \text{End}_{\text{O}(W,B)} (\omega [ V\otimes W]),$$
consisting of endomorphisms of the restricted oscillator representations 
considered as the tensor products
\rref{ResSpAsTens}, \rref{ResOrthoAsTens}
which do not factor through a lower degree tensor product of factors
$\omega_a [V]$, $\text{Res}_{\text{O}(W,B)} (\omega [\F_q^2 \otimes W])$, respectively.
In the symplectic case, in the language of Gurevich and Howe,
$\text{End}_{\text{Sp}(V)}^{top} (\omega [ V\otimes W])$
is precisely the endomorphism algebra of the part of
$\text{Res}_{\text{Sp}(V)}(\omega [V\otimes W])$ with highest $\otimes$-rank.

\vspace{3mm}

The purpose of the next two subsections is to prove the following
\begin{proposition}\label{GroupInTop}
Suppose $V$ is a symplectic space of dimension $2N$ and $W$ is an $n$-dimensional space
with symmetric bilinear form $B$. Write $h_W$ for the maximal dimension of an isotropic subspace
of $W$ with respect to $B$.
\begin{enumerate}
\item\label{GroupInTopSymp} If $(\text{Sp}(V), \text{O}(W,B))$ is in the symplectic stable range, then
\beg{GeometricSubAlgOWB}{\C \text{O} (W,B) \subseteq \text{End}_{\text{Sp}(V)}^{top} (\omega [ V\otimes W])}

\vspace{3mm}

\item \label{GroupInTopOrtho}
If $(\text{Sp}(V), \text{O}(W,B))$ is in the orthogonal stable range, then
\beg{GeometricSubAlgSpV}{\C \text{Sp}(V) \subseteq \text{End}_{\text{O}(W,B)} (\omega [ V\otimes W]),}
\end{enumerate}
In fact, in both cases, this group algebra will act on $\omega [ V\otimes W]$
precisely according to the representation action
as a subgroup of $\text{Sp}(V)\times \text{O}(W,B) \subseteq \text{Sp}(V \otimes W)$.
\end{proposition}

In this subsection, we will prove part \rref{GroupInTopSymp} of this proposition, for
$(\text{Sp}(V), \text{O}(W,B))$ in the symplectic stable range. The case is simpler since the orthogonal group
has a simpler set of generators that the symplectic group. 
Due to the complexity of the generators of $\text{Sp}(V)$ and their corresponding elements in
$\text{End}_{\text{O}(W,B)} (\omega [ V\otimes W])$ as the fixed point space $(\C V\otimes W)^{\text{O}(W,B)}$, 
we will treat part \rref{GroupInTopOrtho} separately
in Subsection \ref{TopSubSectOrtho} below.

\vspace{3mm}

\begin{proof}[Proof of Proposition \ref{GroupInTop}, part \rref{GroupInTopSymp}]
In the orthogonal group $O(W,B)$, consider the set of generators consisting
of 
Our approach will be to find generators corresponding to

Let us consider, for $\lambda = (\lambda_1, \dots , \lambda_n)$
a $1 \times n$ matrix over $\F_q$ in reduced row echelon form, the element
$$f_\lambda := \sum_{v \in V} (\lambda_1 v, \dots , \lambda_n v) = \sum_{v\in V} v\otimes \lambda \in
(\C V \otimes W)^{\text{Sp}(V)}$$
(we may consider $\lambda \in \F_q^n = W$.

\begin{lemma}
For $\lambda \in W $ a $1 \times n$ matrix in reduced row echelon form such that
$B (\lambda, \lambda) \neq 0$, the element $f_\lambda / q^N$ is a reflection
\beg{flambdasgivereflections}{(f_\lambda/ q^N) \star (f_\lambda/ q^N) = (0) = Id_{\omega [ V\otimes W]}}
\end{lemma}

\vspace{3mm}

\begin{proof}
Suppose $\lambda \in W$ is a $1 \times n$ matrix in reduced row echelon form.
Then, applying \rref{FullGeneralityStarVOW}, we get that
\beg{flambdastarflambdastart}{
\begin{array}{c}
\displaystyle f_\lambda \star f_\lambda = \sum_{u,v \in V} (u \otimes \lambda) \star (v \otimes \lambda) =\\
\displaystyle \sum_{u,v \in V} \psi (\frac{B (\lambda , \lambda)\cdot  S(u,v)}{2}) ((u+v) \otimes \lambda).
\end{array}}
If $B (\lambda, \lambda) \neq 0$, substituting $u' = u+ v$ gives
$$f_\lambda \star f_\lambda = \sum_{u,v \in V} \psi (\frac{  S(u',v) \cdot B (\lambda , \lambda)}{2}) \cdot
(u' \otimes \lambda).$$
Collecting terms, for each non-zero $u' \in V$, the coefficient of $u' \otimes \lambda$ is a non-trivial
sum of characters, giving $0$. At $u' = 0$, each $v \in V$ contributes coefficient $\psi (0) = 1$. Hence,
$$f_\lambda \star f_\lambda = q^{2N} \cdot (0) \in (\C V \otimes W)^{\text{Sp}(V)},$$
giving \rref{flambdasgivereflections}.

\end{proof}

\vspace{3mm}

Now we verify that these proposed generators $f_\lambda$ for $\lambda$ satisfying 
$B(\lambda, \lambda ) \neq 0$
really do act as the reflection
elements of $O(W,B)$ across the orthogonal hyperspaces in $W$ to 
each $\lambda$.
To see this we use the Schr\"{o}dinger model for each factor
$\omega_{a_i}[V]$ in $\text{Res}_{\text{Sp}(V)} (\omega [ V\otimes W]) = \omega_{a_1} [V] \otimes \dots\otimes \omega_{a_n} [V]$ and use
 \rref{TrueStarAction} to each $\omega_{a_i}[V]$.

\vspace{3mm}

\begin{claim}\label{ClaimFLambdaAction}
Consider a $\lambda \in W$ in reduced row echelon form such that $B (\lambda, \lambda) \neq 0$.
For $x\in \Lambda^-$, $w\in W$, we have
\beg{GeomActionFLambda}{\frac{f_\lambda ( x \otimes w)}{q^N} = x\otimes (w- 2\frac{w B \lambda^T}{B (\lambda, \lambda)} \lambda).}
\end{claim}

\begin{proof}[Proof of Claim \ref{ClaimFLambdaAction}]
Fix $\lambda \in W$ in reduced row echelon form such that $B (\lambda, \lambda) \neq 0$.
Consider an element $x \in \Lambda_-$, $w= (w_1,\dots , w_n) \in W= \F_q^n$.
Writing out $f_\lambda$, we have
\beg{writeoutflappdtows}{\begin{array}{c}
f_{\lambda} (w \otimes x) = f_\lambda (w_1  x, \dots , w_n x) =\\[1ex]
\displaystyle \sum_{v_\pm \in \Lambda_\pm}
(\lambda_1v_+ + \lambda_1 v_-, \dots , \lambda_n v_+ + \lambda_n v_-) (x_1, \dots , x_n)
\end{array}}
Recalling \rref{TrueStarAction}, for each $i = 1, \dots , n$, applying
an element $\lambda_i v_+ + \lambda_i v_-$ of $V$ acts on $x_i \in \Lambda_-$, 
according to the structure of $\omega_{a_i} [V]$, by
$$\begin{array}{c}
(\lambda_i v_+ + \lambda_i v_-) (w_i x) = \\
\displaystyle \psi_{a_i} (\lambda_i w_i \cdot S (v_+, x) + \lambda_i^2 \frac{S(v_+, v_-)}{2}) (\lambda_i v_- + w_i x).
\end{array}$$
Therefore, in the tensor product $\omega_{a_1}[V] \otimes \dots \otimes \omega_{a_n}[V]$,
a term of \rref{writeoutflappdtows}
for a choice of $v_+ \in \Lambda_+$, $v_- \in \Lambda_-$
reduces to the $n$-tuple of vectors
\beg{Nocoeffvectorwriteoutlv-+w}{(\lambda_1 v_- + w_1 x, \dots , \lambda_n v_- + w_nx)
= \lambda \otimes v_- + w \otimes x,}
writing $\lambda \otimes v_- = (\lambda_1 v_- , \dots , \lambda_n v_-)$ 
as a $1 \times n$ matrix 
with entries in $\Lambda_-$,
multiplied by the coefficient
\beg{Coeffofvectorwriteoutlv-+w}{\begin{array}{c}
\displaystyle
\psi (\sum_{i = 1}^n a_i (\lambda_i w_i\cdot S(v_+, x) + \lambda_i^2 \frac{S(v_+, v_-)}{2}))
 =\\
\displaystyle \psi_1 (S (v_+, \sum_{i=1}^n a_i ( \lambda_i w_i x + \frac{\lambda_i^2 v_-}{2}))).
\end{array}}
Now, since the term \rref{Nocoeffvectorwriteoutlv-+w} does not depend on $v_+$, for a fixed
$v_-$, its coefficient in \rref{writeoutflappdtows} is the sum over all $v_+ \in \Lambda_+$
of terms \rref{Coeffofvectorwriteoutlv-+w}. Again, since linear sums of characters are
$0$, this gives that the coefficient of \rref{Nocoeffvectorwriteoutlv-+w} vanishes unless
\beg{NoZerov-conditionflappdtowrmk}{\sum_{i=1}^n a_i ( \lambda_i w_i x+ \frac{\lambda_i^2 v_-}{2}) = 0,}
in which case for every choice of $v_+$ \rref{Coeffofvectorwriteoutlv-+w} is $1$,
and therefore the coefficient is $q^N$. Now we may rewrite
\rref{NoZerov-conditionflappdtowrmk} as
$$
\frac{(a_1 \lambda_1^2 + \dots + a_n \lambda_n^2) v_-}{2} +(\sum_{i=1}^n a_i \lambda_i w_i)
x = 
\frac{B(\lambda, \lambda)}{2} v_- +  (w B \lambda^T)\cdot x, 
$$
using $w B \lambda^T=\sum_{i=1}^n a_i \lambda_i w_i $, considering $\lambda$ as a $1 \times n$ matrix,
$B$ as an $n \times n$ matrix, and $w^T = (w_1, \dots , w_n)^T$ as an $n \times 1$ matrix.
Therefore, the only surviving term \rref{Nocoeffvectorwriteoutlv-+w}
occurs with coefficient $q^N$ for 
$$v_- = -2\frac{w B \lambda^T}{B(\lambda, \lambda)} x.$$
Substituting this in \rref{Nocoeffvectorwriteoutlv-+w}, we get that
$$f_\lambda (w \otimes x) = q^N \cdot (x\otimes (w - 2\frac{w B \lambda^T}{B(\lambda, \lambda)}\lambda)),$$
giving \rref{GeomActionFLambda}.
\end{proof}

\vspace{5mm}

Since these reflections generate $\text{O}(W,B)$, this claim therefore implies that
for any $\phi \in \text{O}(W,B)$, it acts on an element $x \otimes w\in \Lambda_- \otimes W$ by
$$x \otimes w \; \mapsto \; x \otimes \phi (w),$$
which is precisely its action as an element of $\text{O}(W,B) \subseteq \text{Sp}(V) \times \text{O}(W,B) \subseteq \text{Sp}(V \otimes W)$.
Therefore the subalgebra generated by the elements $f_\lambda$ for $\lambda \in W$
in reduced row echelon form such that $B(\lambda, \lambda) \neq 0$ is isomorphic to 
the group algebra
$$\C \text{O}(W,B) \subseteq \text{End}_{\text{Sp}(V)} (\omega [ V\otimes W])$$
with each generating group element acting according to the geometric representation
action on $\omega [V\otimes W]$. In particular, every generating element of $O(W,B)$
is an automorphism of $\omega [ V\otimes W]$, and therefore cannot factor through any
smaller tensor power of oscillator representations $\omega_a [V]$.
Hence, it is in the top part of the endomorphism algebra and we have \rref{GeometricSubAlgOWB}.
\end{proof}

\vspace{5mm}

\noindent {\bf Remark:} One may ask what the significance of 
an element $f_\lambda$ is for $\lambda \in W$ in reduced row echelon form such that
$B(\lambda, \lambda)$.

\begin{proposition}
Suppose $\lambda_1, \dots , \lambda_k \in W$ are $1 \times n$ matrices in reduced row echelon form
which are linearly independent and such that their
span $\langle\lambda_1, \dots , \lambda_k \rangle \subseteq W$
is an isotropic subspace with respect to $B$. Then $f_{\lambda_i}/ q^{2N}$ are all commuting
idempotents and
\beg{Imageoff1dotsfkis}{\text{Im} (\frac{f_{\lambda_1} \star \dots \star f_{\lambda_k}}{q^{2Nk}})
\cong \text{Res}_{\text{Sp}(V)} (\omega [ V\otimes W[-k]])}
considering the restriction of $\omega [ V\otimes W[-k]]$ to $\text{Sp}(V)$ along the inclusion
$$\text{Sp} (V) \hookrightarrow \text{Sp}(V) \times \text{O}(W[-k], B[-k]) \hookrightarrow \text{Sp}(V \otimes W[-k]).$$
\end{proposition}

\begin{proof}

Now in the case when $B (\lambda, \lambda) =0$, in \rref{flambdastarflambdastart},
the coefficient of every term is trivial. Substituting $u' = u +v$ gives
$$f_\lambda \star f_\lambda = \sum_{u', v\in V} u' \otimes \lambda = q^{2N} f_\lambda,$$
and so $f_\lambda/ q^{2N}$ is an idempotent.
Now for any $\lambda, \mu \in W$ considered as $1 \times n$ matrices in reduced row echelon
form, we have
$$\begin{array}{c}
\displaystyle f_\lambda \star f_\mu = \sum_{u,v \in V} (u \otimes \lambda) \star (v \otimes \mu)=\\
\displaystyle \sum_{u,v \in V} \psi (\frac{B (\lambda, \mu) \cdot S(v,w)}{2}) \cdot ( u\otimes \lambda + v\otimes \mu).
\end{array}$$
In particular, this immediately gives that for $\lambda, \mu$ such that $\langle \lambda, \mu\rangle$
is isotropic with respect to $B$ then $f_\lambda$ and $f_\mu$ commute.

It remains to prove that \rref{Imageoff1dotsfkis} holds. Fix $k \leq h_W$.
For simplicity, without loss of generality, let us write
\beg{GenBFormRmk}{B = \begin{pmatrix}
1 & 0\\
0 & -1
\end{pmatrix}^{\oplus k}\oplus B[-k], \hspace{5mm} B[-k] = \begin{pmatrix}
a_{2k+1} & 0 & \dots & 0\\
0 & a_{2k+2} & & 0\\
\vdots & & & \vdots\\
0 & 0 & \dots & a_n
\end{pmatrix}}
(i.e. $B$ is diagonal, and the first $2k$ diagonal entries are alternating $\pm 1$).
Then \rref{ResSpAsTens} gives
\beg{ksSeperateForIdemFLambs}{\begin{array}{c}
\text{Res}_{\text{Sp}(V)} (\omega [ V\otimes W]) =\\
 (\omega[ V] \otimes \omega_{-1}[V])^{\otimes k}
\otimes \text{Res}_{\text{Sp}(V)} ( \omega[ V\otimes W[-k]]).
\end{array}}

Let us consider an individual $\omega[V] \otimes \omega_{-1} [V]$ factor, i.e.
consider the case of $n=2$ and $B = \begin{pmatrix} 1 & 0 \\ 0 & -1\end{pmatrix}$.
In this case, $\lambda = (1,1)$ gives an idempotent
$$f_{(1,1)}/q^{2N} = \frac{1}{q^{2N}} \sum_{v\in V} (v,v).$$
Now the only term of $f_{(1,1)}/ q^{2N}$ contributing to its trace is $v = 0$, giving
$$\text{dim} (\text{Im} (f_{(1,1)}/q^{2N})) = tr ( f_{(1,1)}/q^{2N}) = 1,$$
and hence $\text{Im} (f_{(1,1)}/q^{2N})$ is the trivial representation of $\text{Sp}(V)$
(since it is the only one-dimensional dimension).

Now let us return to the general case of $n$ and $B$ of the form \rref{GenBFormRmk}.
Consider $\lambda_i$ to be the $n$-tuple
with all entries $0$ except for the $(2i-1)$th and $2i$th which are prescribed to be $1$.
Then as an element of \rref{ksSeperateForIdemFLambs}, each $f_{\lambda_i}$
is a tensor product of $Id_{\omega [V] \otimes \omega_{-1}[V]}$ factors, except for the
$i$th one, which is replaced by a factor $f_{(1,1)} \in End (\omega [ V] \otimes \omega_{-1}[V])$,
with $Id_{\omega[ V\otimes W[-k]]}$. Therefore, we may consider
$$\frac{f_{\lambda_1} \star \dots \star f_{\lambda_k}}{q^{2Nk}} = (f_{(1,1)})^{\otimes k}
\otimes Id_{\omega [ V\otimes W[-k]]},$$
which then has image 
$$1^{\otimes k} \otimes \text{Res}_{\text{Sp}(V)} (\omega [ V\otimes W[-k]])= \text{Res}_{\text{Sp}(V)} (\omega [ V\otimes W[-k]]).$$

To get the statement in the general case of $\lambda_1, \dots , \lambda_k$,
note that the images of idempotents $(f_{\lambda_1} \star \dots \star f_{\lambda_k}) / q^{2Nk})$ for $\lambda_i \in W$ generating a $k$-dimensional isotropic subspace are isomorphic,
since any basis of any $k$-dimensional istorpic subspace can be transformed into
any other using an orthogonal group element, by Witt's Theorem. Hence, all images
of such $(f_{\lambda_1} \star \dots \star f_{\lambda_k}) / q^{2Nk})$ are isomorphic.
\end{proof}

\vspace{5mm}

\subsection{Generators of the symplectic group}\label{TopSubSectOrtho}

The purpose of this subsection is to prove Proposition
\ref{GroupInTop}, part \rref{GroupInTopOrtho}
in the case when the reductive dual pair $(\text{Sp}(V), \text{O}(W,B))$
in $\text{Sp} (V\otimes W)$ is in the orthogonal stable range.

\begin{proof}[Proof of Proposition \ref{GroupInTop}, \rref{GroupInTopOrtho}]
We use the same method as in 
the previous subsection, by finding elements of
$\C (V\otimes W)^{\text{O}(W,B)}$ and proving they act on $\omega [ V\otimes W]$
according to the representation-theory
action of corresponding generators of $\text{Sp}(V)$.

Let us again recall how to apply an endomorphism in 
We will again use the Schr\"{o}dinger model of $ \omega [ V \otimes W]$. Let us
write $V = \Lambda^+ \oplus \Lambda^-$ for $V$'s decomposition into complementary Lagrangians
with respect to $S$. According to \rref{TrueStarAction},
writing an element of the algebra
$(\C ( V \otimes W), \star)$ as $(v_1^+ + v_1^-, \dots , v_n^+ +v_n^-)$ for $v_i^\pm \in \Lambda^\pm$,
and writing an element of $ \omega [  V\otimes W] = \C W \otimes \Lambda^-$ as
$(x_1 , \dots , x_n)$ for $x_i \in \Lambda^-$, we have
\beg{SymplecticTwistedSchroAct}{\begin{array}{c}
(v_1^+ + v_1^-, \dots , v_n^++v_n^-) \cdot (x_1 , \dots , x_n) =\\[1ex]
\displaystyle \psi (\sum_{j=1}^n a_i \cdot (S (v_i^+ , x_i) +  \frac{S(v_i^+, v_i^-)}{2})) \cdot (v_1^- + x_1, \dots , v_n^- +x_n)
\end{array}}
(where $\psi$ denotes the non-trivial additive character corresponding to our choice of oscillator
representation $\boldsymbol \omega$).

To write this in terms of the symmetric bilinear form $B$ and vectors $u\in W$,
let us fix bases of the Lagrangians $\Lambda^+$, $\Lambda^-$ such that, with respect to the
basis $e_1^+ , \dots , e_N^+, e_1^- , \dots , e_N^-$ of $V$,
the symplectic form $S$ is 
$$\begin{pmatrix}
0 & I\\
-I & 0
\end{pmatrix}.
$$
Then, alternatively, writing an element of $W \otimes V$
as $(z_1^+ , z_1^-, \dots , z_{N}^+ , z^-_N)$ for $z_i^\pm \in W \otimes \F_q\{e_i^\pm\}$,
and an element of $W \otimes \Lambda^-$ as 
$(u_1, \dots , u_N)$ for $u_i \in W \otimes  \F_q\{e_i^-\}$, we have
\beg{zipmonuiSchroModelAct}{\begin{array}{c}
(z_1^+, z_1^-, \dots , z_{N}^+ , z^-_N) \cdot (u_1, \dots , u_N) = \\[1ex]
\displaystyle
\psi (\sum_{i=1}^N B( z_i^+, u_i) + \frac{B(z_i^+, z_i^-)}{2}) \cdot (u_1 + z_1^-, \dots , u_N + z_N^-)  
\end{array}}

\vspace{5mm}

Now we can pick elements of $\C (V\otimes W)^{\text{O}(W,B)}$ designed
to act as generators of $\text{Sp}(V)$.

Let us first consider the case when $dim (V) = 2$. 
In this case, we may reduce \rref{zipmonuiSchroModelAct} to
$$(z^+, z^-) \cdot (u) = \psi (B (z^+, u) + \frac{B (z^+, z^-)}{2}) \cdot (u + z^-).$$

From this perspective, the action of the matrices in $SL_2 (\F_q)$
$$
\begin{pmatrix}
0 &1\\
-1 & 0
\end{pmatrix},
\hspace{10mm}
\begin{pmatrix}
1 & 0\\
t & 1
\end{pmatrix}
$$
on the oscillator representation should correspond to transformations
\beg{ActionsFromOsc}{\resizebox{0.87\textwidth}{!}{$\displaystyle (u) \mapsto \frac{1}{(-q)^{n/2}} \sum_{w \in W} \psi (B (-u, w)) \cdot (w), \hspace{3mm}
(u) \mapsto \psi (\frac{t B (u,u )}{2}) \cdot (u)$}}
respectively. The matrices
$$\begin{pmatrix}
s & 0\\
0 & 1/s
\end{pmatrix}$$
act by $(u) \mapsto \epsilon (s) \cdot (s \cdot u)$, for $s\in \F_q^\times$.
Now consider, for example, the operators given by the action of
$$g_{t} = \frac{\epsilon (t) }{(-q)^{n/2}} \cdot \sum_{z \in W} \psi (\frac{t B(z,z)}{2}) \cdot (z, t z) \in \C ( W \otimes V)^{\text{O} (W)}$$
for $t \in \F_q^\times$. Applied to $u \in W$, these endomorphisms give
$$ g_t \cdot (u) = \frac{\epsilon(t)}{(-q)^{n/2}} \cdot \sum_{z \in W} \psi ( B (z,u) + B (z, tz)) \cdot (u+ tz),$$
which, replacing $w = u+tz$, can be simplified to
$$\begin{array}{c}
\displaystyle\frac{\epsilon(t)}{(-q)^{n/2}} \cdot \sum_{w \in W} \psi (B(\frac{w-u}{t}, w)) \cdot (w) =\\
\displaystyle \frac{\epsilon (t)}{(-q)^{n/2}} \cdot \sum_{w \in W} \psi (\frac{B(w,w)}{t}) \cdot \psi ( B ( \frac{-u}{t} , w)) \cdot (v). 
\end{array}$$
Therefore, each $g_t$
corresponds to the group action of the composition of matrices 
\beg{PoissonDists}{
\begin{pmatrix}
1 & 0\\
2/t & 1
\end{pmatrix}
\begin{pmatrix}
0 & 1\\
-1 & 0
\end{pmatrix}
\begin{pmatrix}
1/t & 0\\
0 & t
\end{pmatrix}
 = \begin{pmatrix}
0 & t \\
-1/t & 2
\end{pmatrix}}
on $\omega [ W \otimes V]$ (note that $\epsilon (t) = \epsilon (1/t)$).
These matrices generate $\text{SL}_2 (\F_q)$.

Now, for general $V$, $\text{dim} (V)= 2N$, we may find these generators 
for all choices of 1-dimensional subspaces in a Lagrangian (and its dual).
This system of group algebras
over $\text{SL}_2 (\F_q)$ corresponding to choices of isotropic $1$-dimensional
subspace of $V$ therefore generate $\text{Sp} (V)$.
Hence, we get
\beg{TopSubalg}{\C \text{Sp} (V) \subseteq \text{End}_{\text{O}(W)} ( \omega [W \otimes V]).}
Additionally, since these endomorphisms encode the geometric action of
$\text{Sp} (V) \subseteq \text{Sp} (W \otimes V)$ on $\omega [ W \otimes V]$
and are, in particular, bijective, they are
inexpressible as compositions factoring through a lower degree tensor power
of $\text{Res}_{\text{O}(W,B)} (\omega [\F_q^2 \otimes W]) \cong \C W^-$.
\end{proof}

\vspace{5mm}

The reason why we use the matrices \rref{PoissonDists} (instead of a
more common set of generators of $SL_2 (\F_q)$, such as \rref{TheUsualGenerators} below)
is due to the fact that their corresponding elements
of $\C (W \otimes \F_q^2)^{O(W)}$ are fairly simple and easy to guess.
While not directly necessary to the logic of the proof of the
results of the above proposition,
it is, however, instructive to write down the explicit formulae for the elements corresponding to
the matricies
\beg{TheUsualGenerators}{\begin{pmatrix}
t & 0 \\
0 & 1/t
\end{pmatrix},
\begin{pmatrix}
0 & 1\\
-1 & 0
\end{pmatrix},
\begin{pmatrix}
1& 0\\
s & 1
\end{pmatrix}
}
for $t, s \in \F_q^\times$.

\vspace{5mm}

To do this, we need to introduce certain
constants, arising from the quadratic sums of characters. They
will depend on our choice of charcter $\psi$.
If we use a different character, the answer may differ by a sign.
Let us write $q= p^\ell$ for $p$ an odd prime, and $\ell \in \N$.
To avoid confusion, we
denote the quadratic multiplicative characters of $\F_q$ and $\F_p$ by
$$\epsilon_{p} : \F_p^\times \rightarrow \{\pm 1\}, \hspace{10mm} \epsilon_q: \F_q^\times \rightarrow
\{\pm 1\},$$
respectively. As usual, we extend these to $0$ by $\epsilon_p (0) = \epsilon_q (0) = 0$.
Denoting the norm of the field extension
by $N_{\F_q/ \F_p} : \F_q^\times \rightarrow \F_p^\times $, we have
\beg{enise}{\epsilon_p \circ N_{\F_q/ \F_p}  = \epsilon_q.}

We may re-write the classical quadratic Gauss sum as
\beg{ReductionToActualGaussSum}{\sum_{n \in \F_p} e^{\frac{2\pi i}{p}
n^2} = \sum_{m \in \F_p} (1+ \epsilon_p (m)) \cdot
e^{\frac{2\pi i}{p}n^2} = 
\sum_{m \in \F_p} \epsilon_p (m) \cdot e^{\frac{2\pi i}{p}n^2}}
(since the linear sum of characters is $0$, and for each $m \in \F_p$, there are exactly
$1 + \epsilon_p (m)$ elements in $\F_p$ whose square is $m$),
which is well-known to equal
\beg{SinglepCase}{\sum_{n \in \F_p} e^{\frac{2\pi i}{p}a \cdot n^2} =
\epsilon_p (a) \cdot \sqrt{ \epsilon_p (-1) \cdot p}.}

Now the same argument as \rref{ReductionToActualGaussSum} can be applied to give
$$\begin{array}{c}
\displaystyle \sum_{x \in \F_q} \psi (x^2) = \sum_{x \in \F_q}e^{\frac{2\pi i}{p} Tr_{\F_q/\F_p} ( x^2)}=\\
\\
\displaystyle \sum_{y \in \F_q} \epsilon_q (y) \cdot e^{\frac{2\pi i}{p} Tr_{\F_q/\F_p} ( y)}.
\end{array}$$
Applying the Hasse-Davenport relation for Gauss sums 
to \rref{SinglepCase},
we get that
$$\sum_{y \in \F_q} \epsilon_q (y) \cdot e^{\frac{2\pi i}{p} Tr_{\F_q/\F_p} ( y)} = (-1)^{\ell+ 1} \cdot \left( \sqrt{ \epsilon_p (-1) \cdot p}\right)^\ell ,$$
which simplifies to give
\beg{Psi1Cases}{
\sum_{x \in \F_q} \psi (x^2)  = (-1)^{\ell +1} \sqrt{ \epsilon_q (-1) \cdot q}
}

Now, for $c \in \F_q^\times$, again since a linear sum of characters vanishes, we have
$$ \sum_{x\in \F_q} \psi ( c\cdot x^2) = \epsilon_q (c) \cdot \sum_{x \in \F_q} \psi (x^2).$$
Combining this with \rref{Psi1Cases}, we find that
$$\begin{array}{c}
\displaystyle \sum_{u \in W} \psi ( c\cdot B (u,u)) =   \sum_{u_1, \dots , u_n \in \F_q} \psi (
\sum_{i =1}^n c \cdot a_i \cdot u_i^2) =\\
\\
\displaystyle \prod_{i=1}^n \sum_{u_i \in W} \psi (c \cdot a_i \cdot u_i^2) = \epsilon_q (c^n \cdot a_1 \dots a_n) \cdot 
(\sum_{x \in \F_q} \psi (x^2))^n =\\
\\
\displaystyle (-1)^{n (\ell + 1)} \cdot \text{disc} (B) \cdot q^{n/2}\cdot \epsilon_q (c)^n  \cdot \epsilon_q(-1)^{n/2} ,
\end{array}
$$
where $\text{disc}(B)$ denotes is discriminant, i.e. $\epsilon_q (\text{det}(B))$.
For notational brevity, we denote these coefficients by
\beg{Konstant}{\begin{array}{c}
\displaystyle K(c):=\sum_{u \in W} \psi ( c\cdot B (u,u))  =\\
 (-1)^{n (\ell + 1)} \cdot \text{disc} (B) \cdot q^{n/2}\cdot \epsilon_q (c)^n  \cdot \epsilon_q(-1)^{n/2}.
\end{array}}

\vspace{3mm}

\begin{proposition}\label{ExplicitFormulaeForGensSL2Prop}
Suppose we are given an orthogonal space $(W,B)$ such that
$(SL_2 (\F_q), \text{O}(W,B))$ forms a reductive dual pair in the orthogonal stable range.
\begin{enumerate}
\item \label{ExplicitFormulaeForGensSL2PropScal} For $t\neq 1\in \F_q^\times$, the matrix
$$\begin{pmatrix}
t & 0\\
0 & 1/t
\end{pmatrix} \in SL_2 (\F_q)
$$
as an element of $\text{End}_{\text{O}(W,B)} (\omega [\F_q^2\otimes W])$ corresponds to
the element
$$\alpha_t := \frac{\epsilon (t) }{q^n} \cdot \sum_{y^+, y^- \in W} \psi (- \frac{t+1}{2(t-1)} \cdot B (y^+, y^-)) \cdot (y^+, y^-)$$
of $(\C V\otimes W)^{\text{O}(W,B)}$.

\item \label{ExplicitFormulaeForGensSL2PropFour} The matrix
$$\begin{pmatrix}
0 & 1\\
-1 & 0
\end{pmatrix}\in SL_2 (\F_q)
$$
as an element of $\text{End}_{\text{O}(W,B)} (\omega [\F_q^2 \otimes W])$ corresponds to the element
$$\resizebox{0.95\textwidth}{!}{$\displaystyle \beta:= \frac{1}{K(1) \cdot (-q)^{n/2}}\sum_{y^+, y^- \in W} \psi \left(\frac{1}{4} (B(y^+, y^+) + B (y^-, y^-))\right)
\cdot (y^+, y^-)$} $$
of $(\C V\otimes W)^{\text{O}(W,B)}$.

\item \label{ExplicitFormulaeForGensSL2PropAdv}
For $s\in \F_q^\times$, the advection matrix
$$\begin{pmatrix}
1 & 0\\
s & 1
\end{pmatrix} \in SL_2 (\F_q)
$$
as an element of $\text{End}_{\text{O}(W,B)} (\omega [\F_q^2 \otimes W])$ corresponds to the element
$$\gamma_s := \frac{1}{K(\frac{-1}{2s})} \sum_{z \in W} \psi (-\frac{1}{2s} B (z,z)) \cdot (z,0)$$
of $(\C V\otimes W)^{\text{O}(W,B)}$.

\end{enumerate}
\end{proposition}

\begin{proof}[Proof of Proposition \ref{ExplicitFormulaeForGensSL2Prop}]

We begin with the proof of (\ref{ExplicitFormulaeForGensSL2PropScal}): 
For an element $u\in W \cong W \otimes \Lambda^-$, for $t\in \F_q^\times$,
we can calculate the following:
\beg{ftappliedtoufirsttry}{\begin{array}{c}
\alpha_t(u)= \\[1ex]
\resizebox{0.8\textwidth}{!}{$\displaystyle \frac{\epsilon (t)}{q^n} \cdot\sum_{y^+, y^- \in W} \psi \left( (\frac{1}{2}- \frac{t+1}{2(t-1)}) \cdot B (y^+, y^-)
+ B (y^+, u)\right)\cdot (y^- + u)$} \\
\\
\displaystyle = \frac{\epsilon (t)}{q^n} \cdot 
\sum_{y^+, y^- \in W} \psi (B (y^+ ,  - \frac{y^-}{t-1} + u)) \cdot(y^- + u).
\end{array}
}
The sum runs over arbitrary choices of $y^+$, meaning that for fixed $u\in W$ and chosen $y^- \in W$,
the coefficient sum
\beg{TorusCoeffLinChar}{\sum_{y^+ \in W} \psi ( B (y^+ , -\frac{y^-}{t-1} + u))}
of the vector $(y^- + u)$ is a linear sum of characters, and is therefore $0$, unless
$u = y^-/(t-1)$,
in which case \rref{TorusCoeffLinChar} is $q^n$. Hence, the only contributing
choice of $y^-$ is $y^- = (t-1) \cdot u$. Therefore, \rref{ftappliedtoufirsttry} simplifies as
$$\alpha_t (u) =\epsilon (t) \cdot \frac{q^n}{q^n} \cdot ( (t-1)\cdot u + u) = (t\cdot u), $$
agreeing with the action of the proposed matrix on the oscillator representation
$\C W $.

\vspace{3mm}

Now we prove (\ref{ExplicitFormulaeForGensSL2PropFour}): 
For $u \in W$, at each choice of $y^+, y^- \in W$, applying the corresponding term of the sum in
$\beta$ (disregarding the coefficient, for now) gives
$$
\begin{array}{c}
\displaystyle \psi \left(\frac{1}{4} (B(y^+, y^+) + B (y^-, y^-))\right)
\cdot (y^+, y^-) (u)=\\
\\
\displaystyle \psi \left(\frac{1}{4} (B(y^+, y^+) + 2 B (y^+, y^-) + B (y^-, y^-)) + B (y^+, u)\right) \cdot (y^- + u)=\\
\\
\displaystyle \psi ( B (\frac{y^+ + y^-}{2}, \frac{y^++y^-}{2}) + B(y^+, u) ) \cdot (y^- + u)
\end{array}
$$
Therefore, we have
\beg{happliedtouFourTrans1sttry}{\begin{array}{c}
\beta(u) =\\[1ex]
\resizebox{0.85\textwidth}{!}{$\displaystyle
\frac{1}{K (1)\cdot (-q)^{n/2}} \sum_{y^+, y^- \in W}\psi ( B (\frac{y^+ + y^-}{2}, \frac{y^+ + y^-}{2}) + B(y^+, u) ) \cdot (y^- + u).$}
\end{array}}
Renaming variables using $z = y^- + u$, we may rewrite this as
\beg{happtouwithznoy-}{\resizebox{0.85\textwidth}{!}{$\displaystyle\frac{1}{K(1) \cdot (-q)^{n/2}}\sum_{y^+, z \in W} \psi ( B ( \frac{y^+ + z -u}{2}, \frac{y^+ + z-u}{2}) + B (y^+, u)) \cdot (z)$}.}
Now we may also notice that
$$\begin{array}{c}
\displaystyle
B ( \frac{y^+ +z -u}{2}, \frac{y^+ +z-u}{2}) + B (y^+, u) = \\
\\
\displaystyle B (\frac{y^+ + z + u}{2}, \frac{y^+ +z+u}{2})
- B(z, u),
\end{array}$$ 
allowing us to rewrite \rref{happtouwithznoy-} as
$$\frac{1}{K(1) \cdot (-q)^{n/2}} \sum_{z,y^+ \in W} 
\psi (B(\frac{y^++z+u}{2}, \frac{y^++z+u}{2}) - B(z,u)) \cdot (z).
$$
Renaming variables using $ w= (y^+ + z+u )/2$ gives
$$ \frac{1}{K(1) \cdot (-q)^{n/2}} \sum_{z,w \in W} \psi (B (w,w)) \psi (- B(z,u)) \cdot (z),$$
which, applying \rref{Konstant}, reduces to
$$\beta (u) =\frac{1}{(-q)^{n/2}} \sum_{z \in W} \psi ( - B(z,u)) \cdot (z),$$
which is precisely the action \rref{ActionsFromOsc}.

\vspace{3mm}

Finally, we prove (\ref{ExplicitFormulaeForGensSL2PropAdv}):
For $u \in W$, $s\in \F_q^\times$,
\beg{gammasu}{
\begin{array}{c}
\gamma_s (u) = \\[1ex]
\displaystyle   \frac{1}{K(\frac{-1}{2s})} \sum_{z \in W} \psi (-\frac{1}{2s} B (z,z)) \cdot (z,0) (u) =\\
\\
\displaystyle  \frac{1}{K(\frac{-1}{2s})}
\sum_{z \in W} \psi ( -\frac{1}{2s} B (z,z) + B (z,u)) \cdot (u).
\end{array}
}
Now we may notice that 
$$
\begin{array}{c}
\displaystyle B (z,u) - \frac{1}{2s} B (z,z) = -\frac{1}{2s} ( - 2s \cdot B (z,u) + B (z,z)) =\\
\\
\displaystyle
-\frac{1}{2s}  \left( B(su-z, su-z) - s^2 \cdot B (u,u)\right) =\\
\\
\displaystyle  -\frac{1}{2s} \cdot B (su -z, su-z) + \frac{s}{2} \cdot B(u,u).
\end{array}
$$
Therefore, substituting $w = su -z$, we can rewrite \rref{gammasu} as
\beg{pleaselettheconstantberight}{
\gamma_s (u) = \frac{1}{K(\frac{-1}{2s})}  \sum_{w \in W} \psi (-\frac{1}{2s}\cdot B (w,w))\cdot
\psi (\frac{s}{2} \cdot B (u,u)) \cdot (u).
}
Since, by definition,
$$\begin{array}{c}
\displaystyle
\sum_{w\in W} \psi (- \frac{1}{2s} \cdot B (w,w)) = K (\frac{-1}{2s}),
\end{array}$$
\rref{pleaselettheconstantberight} then reduces to
$$\gamma_s (u)= \psi (\frac{s}{2} \cdot B(u,u)) \cdot (u),$$
agreeing precisely with \rref{ActionsFromOsc}.

\end{proof}

\vspace{3mm}

It may also be helpful to compute some examples of $\star$ applied to these 
elements, and see how it recovers matrix multiplication (especially to see the relationship
between $g_t$, $\alpha_t$, $\beta$, and $\gamma_{2/t}$).
We do an example of such a computation in the Appendix.

\vspace{5mm}

\section{Combinatorics}\label{CombinSection}

In the previous section, we were able to find a subalgebra in the endomorphism.
To set up for an inductive argument, we therefore need to prove some combinatorial
identities about the relationship between $\omega [V\otimes W]$ and lower $\omega[ V\otimes W[-\ell]$ (resp. $\omega [ V[-\ell] \otimes W]$) for the symplectic (resp. orthogonal) stable statements.
We will also need to prove that the dimensions of the endomorphism algebras
of the restricted oscillator representations match the dimensions of their claimed decompositions
into matrix algebras of group algebras. This is the purpose of this section:

\begin{proposition}\label{HomUevenDimsCor}
We use the notation in the assumptions of Theorem \ref{TheoremGenuine}. 
\begin{enumerate}
\item \label{SympPartUnevenProp} If $(\text{Sp}(V), \text{O}(W,B))$
is a reductive dual pair in the symplectic stable range, then for every
\beg{DimensionUnevenHomsPropSymp}{\begin{array}{c}
\text{dim} (\text{Hom}_{\text{Sp}(V)} (\omega [ V \otimes W[-\ell]], \omega [ V\otimes W])) =\\[1ex]
\resizebox{0.87\textwidth}{!}{$\displaystyle\sum_{k = 0}^{h_W -\ell} \frac{|\text{O}(W[-\ell],B[-\ell])|}{|P^{B[-\ell]}_{k}|} \cdot \frac{|\text{O}(W,B)|}{|P^B_{ k+\ell}|} \cdot |\text{O} (W[-(k+\ell)], B[-(k+\ell)])|.$}
\end{array}}

\item \label{OrthoPartUnevenProp} If $(\text{Sp}(V), \text{O}(W,B))$
is a reductive dual pair in the orthogonal stable range, then
\beg{DimensionUnevenHomsPropOrtho}{\begin{array}{c}
\text{dim} (\text{Hom}_{\text{Sp}(V)} (\omega [ V \otimes W[-\ell]], \omega_B)) =\\[1ex]
\displaystyle \sum_{k = 0}^{N} \frac{|\text{Sp}(V[-\ell])|}{|P_k^{V[-\ell]}|} \cdot \frac{|\text{Sp}(V)|}{|P^V_{ k+\ell}|} \cdot |\text{Sp} (V[-(k+\ell)])|.
\end{array}}
\end{enumerate}
\end{proposition}

In particular, plugging in $\ell =0$ gives 
\begin{corollary}\label{Combinatoricsdimmatchell=0topfinal}
We use the notation in the assumptions of Theorem \ref{TheoremGenuine}. 
\begin{enumerate}
\item If $(\text{Sp}(V), \text{O}(W,B))$ is in the symplectic stable range then
\beg{SympStablEndDimCombin}{\begin{array}{c} \text{dim} (\text{End}_{\text{Sp}(V)} (\omega [V\otimes W])) =\\ \displaystyle \sum_{k=0}^{h_W} |O (W,B)/ P_{k}^{W,B}|^2 \cdot |O(W[-k], B[-k])|
\end{array}}

\item If $(\text{Sp}(V), \text{O}(W,B))$ is in the orthogonal stable range then
\beg{OrthoStablEndDimCombin}{\text{dim} (\text{End}_{\text{O}(W,B)} (\omega [V\otimes W])) = \sum_{k=0}^{N} |\text{Sp}(V)/ P_k^V|^2 \cdot |\text{Sp}(V[-k])|}
\end{enumerate}
\end{corollary}
\noindent In other words, this result also combinatorially verifies that the dimensions of both sides
of \rref{SympStableEndDecompThm}
(resp. \rref{OrthoStableEndDecompThm}) in part \rref{SymTheoremGenuine}
(resp. part \rref{OrthoTheoremGenuine}) of Theorem \ref{TheoremGenuine} match.

\vspace{3mm}

In Subsection \ref{KeyCombinLemmaSubsect},
we state a combinatorial lemma on Gaussian binomial coefficients
and use it to derive Proposition \ref{HomUevenDimsCor}.
In Subsection \ref{CombinLemmaProofSubsect}, we prove the lemma.

\vspace{5mm}

\subsection{A key lemma}\label{KeyCombinLemmaSubsect}
Let us write $Q_n = q^n +1$. Recall the Gaussian binomial coefficients
\beg{GaussianCoeffDefn}{{a \choose b}_q := \frac{(q^a -1) \cdot (q^{a-1} -1) \dots (q^{a-b+1} -1)}{(q^b -1) \cdot (q^{b-1} -1) \dots (q-1)}.}
The purpose of this section is to prove the following
\begin{lemma}\label{CombinatorialLemma}
For every $p \in \N_0$, for every $r>b \in \N_0$, 
\beg{CombLemmaGenStat}{\begin{array}{c}
Q_{r+p}\cdot Q_{r-1+p} \dots Q_{b+p} =\\[1ex]
\displaystyle \sum_{a=0}^p q^{a (b+a-1)} \cdot {r-b+1 \choose a}_q 
\cdot \prod_{i=p-a+1}^p (q^i -1) \cdot \prod_{j=b+a}^{r} Q_j
\end{array}}
\end{lemma}

\vspace{3mm}

\noindent Assume Lemma \ref{CombinatorialLemma} holds.
\begin{proof}[Proof of Proposition \ref{HomUevenDimsCor}]
First, let us process the left hand sides of \rref{DimensionUnevenHomsPropSymp} and
\rref{DimensionUnevenHomsPropOrtho}.
Recalling \rref{ResSpAsTens} and
that the tensor product of an oscillator representation with its dual oscillator representation
is the permutation representation, for $0 \leq \ell \leq h_W$, we get
$$\text{Res}_{\text{Sp}(V)} (\omega [ V\otimes W]) = \text{Res}_{\text{Sp}(V)} (\omega [ V\otimes W[-\ell]]) \otimes (\C V)^{\otimes \ell}.$$
Therefore, for $0 \leq \ell \leq h_W$, we have
$$
\text{Hom}_{\text{Sp}(V)} (\omega [V \otimes W[-\ell]], \omega [ V\otimes W]) \cong
\text{Hom}_{\text{Sp}(V)} (1, \C (V^{\oplus (n-\ell)}))$$
Similarly, recalling \rref{ResOrthoAsTens} and \rref{PermOWB}, for $0 \leq \ell \leq N$, we have
$$
\text{Hom}_{\text{O}(W,B)} (\omega [ V[-\ell]\otimes W], \omega [V\otimes W]) \cong \text{Hom}_{\text{O}(W,B)} (1, \C (W^{\oplus 2N-\ell}))
$$
Thus, we may rewrite the left hand side of \rref{DimensionUnevenHomsPropSymp} as
\beg{OddLHSCORStateFirst}{\text{dim} (\text{Hom}_{\text{Sp}(V)} (\omega [ V\otimes W[-\ell]], \omega [ V\otimes W])) = 2 Q_{1} \dots Q_{n-\ell -1}}
and the left hand side of \rref{DimensionUnevenHomsPropOrtho}
\beg{OddLHSCORStateFirstOrtho}{\text{dim} (\text{Hom}_{\text{O}(W,B)} (\omega [ V[-\ell]\otimes W], \omega [ V\otimes W])) = Q_{1} \dots Q_{2N-\ell}}
by Lemma \ref{FixedPointCountingLemma}. In particular, also note that \rref{OddLHSCORStateFirstOrtho}
is precisely half of \rref{OddLHSCORStateFirst} in the case of $n =2N+1$.

\vspace{3mm}

To reduce the right hand sides of \rref{DimensionUnevenHomsPropSymp} and \rref{DimensionUnevenHomsPropOrtho},
we need expressions for the group orders of finite orthogonal groups, finite
symplectic groups, and their parabolic quotients. First, recall that
\beg{GroupOrders}{\begin{array}{c}
\displaystyle |\text{O}_{2m+1} (\F_q)|= 2 q^{m^2} \prod_{i=1}^m (q^{2i}-1)\\
\displaystyle |\text{O}_{2m}^\pm (\F_q)|= 2q^{m(m-1)} (q^m \mp 1) \prod_{i=1}^{m-1} (q^{2i}-1) \\
\displaystyle |\text{Sp}_{2N} (\F_q)| = |\text{O}_{2N+1} (\F_q)|/2 = q^{N^2} \prod_{i=1}^N (q^{2i}-1)
\end{array}}
(recalling the notation that for $n =2m$ even, $\text{O}_{2m}^+ (\F_q)$ denotes the orthogonal group
$\text{O}(W,B)$ when $h_W = m$ and $\text{O}_{2m}^- (\F_q)$ denotes $\text{O}(W,B)$ when $h_W = m-1$).

To compute the orders of the parabolic quotients of orthogonal and symplectic groups first note that
$$|\text{O}(W, B)/P^B_{ k}| = \frac{|\text{O}(W, B)/P^B_{ 1}| \cdot |\text{O}(W[-1],B[-1])/P^{B[-1]}_{ k-1}|}{|\mathbb{P}^{k-1} (\mathbb{F}_q)|},$$
and therefore,
\beg{GORecursion}{|\text{O}(W, B)/P^B_{ k}| = \frac{\displaystyle \prod_{\ell = 1}^{k-1} |\text{O}(
W[-\ell], B[-\ell])/P^{B[-\ell]}_{ 1}|}{\displaystyle \prod_{\ell = 1}^{k-1} |\mathbb{P}^{\ell} (\mathbb{F}_q)|}.}
It is also well known that for any symmetric bilinear form $B$ on $\mathbb{F}_q^n$,
the number of elements in $\text{O}(B)/P^B_{ 1}$,
which is the set of points the quadric defined by $B$ in $\mathbb{P}^{n-1} (\mathbb{F}_q)$, is
\beg{PointsOfQuadricInPn-1}{|\text{O}(W,B)/P^B_{ 1}| =  \begin{cases}
\displaystyle \frac{q^{n-1}-1}{q-1} & n \text{ odd} \\
\\
\displaystyle \frac{q^{n-1}-1}{q-1}\pm q^{(n-2)/2} & \begin{array}{l}
n \text{ even, where}\\
 \text{O}(W,B) = \text{O}_n^\pm (\F_q).
\end{array}
\end{cases}}

\vspace{3mm}

Now the , and in fact, it turns out that
$$|\text{Sp}_{2N} (\F_q) / P_k^{\F_q^{2N}}| = |\text{O}_{2N+1} (\F_q)/ P_{k}^{\F_q^{2N+1}}|$$
(again this can be interpreted as caused by the fact that the symplectic groups are dual to the
special orthogonal groups). In particular, combined with \rref{GroupOrders},
this allows us to rewrite
the right hand side of \rref{DimensionUnevenHomsPropOrtho}
entirely in terms of orders of odd orthogonal groups and their parabolic quotients. In fact,
we find that this right hand side is precisely half of the right hand side.
Since we observed the same effect for the left hand side at the beginning of the proof,
we therefore find that \rref{DimensionUnevenHomsPropOrtho} is,
as an expression, exactly \rref{DimensionUnevenHomsPropSymp}
at $n = 2N+1$ (in which case $N= h_W$, giving the same range of $\ell$), multiplied by $1/2$.
Thus, it suffices to prove part \rref{SympPartUnevenProp} of the proposition.

\vspace{3mm}

\noindent {\bf Case 1:} $n$ is odd.
Let us write $n=2m +1$, in which case
$m$ must be the number of hyperbolics $h_W$ in $B$. In this case, for $k=0, \dots , m-\ell$, 
combining \rref{GORecursion} and \rref{PointsOfQuadricInPn-1},
$$\frac{|\text{O}(W,B)|}{|P^B_{k+\ell}|}= \frac{(q^{2m} -1) \cdot (q^{2(m-1)} -1) \dots (q^{2(m-k-\ell+1)}-1)}{(q^{k+\ell}-1) \cdot (q^{k+\ell-1} -1) \dots (q-1)}.$$
We may rewrite this as
$$\frac{|\text{O}(W,B)|}{|P^B_{k+\ell}|} = {m \choose k+\ell }_q \cdot \prod_{i=m-k-\ell+1}^m (q^i+1).$$
Similarly,
$$\frac{|\text{O}(W[-\ell],B[-\ell])|}{|P^{B[-\ell]}_{ k}|} = \frac{(q^{2(m-\ell)} -1) \dots (q^{2(m-k-\ell +1)} -1)}{(q^k -1) \dots (q-1)}.$$
Applying \rref{GroupOrders}, we also have
$$\begin{array}{c}
\displaystyle |\text{O}(W[-k-\ell],B[-k-\ell]) | = 2 q^{(m-k-\ell)^2} \prod_{j=1}^{m-k-\ell} (q^{2k} - 1) =\\
\vspace{1mm}
\displaystyle 2 q^{(m-k-\ell)^2} \prod_{j=1}^{m-k-\ell} (q^{j} - 1)(q^j+1) .
\end{array}$$

Combining this with \rref{OddLHSCORStateFirst}, the statement reduces to 
\beg{LiteralRestatementOddCase}{\begin{array}{c}
2Q_{2m-\ell} \dots Q_1 = \\[1ex]
\displaystyle \sum_{k = 0}^{m-\ell} 2q^{(m-k-\ell)^2} \cdot {m \choose k+\ell}_q \cdot
\prod_{j=m-k-\ell +1}^m Q_j  \cdot \frac{\displaystyle\prod_{i = 1}^{m-\ell} (q^i-1) \cdot Q_i}{(q^k -1) \dots (q-1)}
\end{array}}
We may divide both sides of \rref{LiteralRestatementOddCase} by $2Q_1 \dots Q_{m-\ell}$, giving
\beg{M+1To2MVersionOddCase}{\begin{array}{c}
Q_{2m-\ell} \dots Q_{m-\ell+1}= \\[1ex]
\displaystyle \sum_{\ell=0}^{m-k} q^{(m-k-\ell)^2} {m \choose k+\ell }_q \cdot
\prod_{i = m-k-\ell+1}^m Q_j \cdot \prod_{j=k +1}^{m-\ell } (q^j -1).
\end{array}}
Replacing
${m \choose k+\ell}_q = {m \choose m-k-\ell}_q,$
this follows exactly from Lemma \ref{CombinatorialLemma} by putting
$r= m$, $p = m-\ell$, $b =1$
and substituting $a = m-k-\ell$.

\vspace{3mm}

\noindent {\bf Case 2:} $n$ is even. Let us write $n=2m$.

{\em Case 2A:} Suppose $B$ decomposes completely into hyperbolics, i.e. $h_W = m$.
In this case, combining \rref{GORecursion} and \rref{PointsOfQuadricInPn-1}, noting first that we may first re-write this case of \rref{PointsOfQuadricInPn-1} as 
$$\begin{array}{c}
\displaystyle \frac{|\text{O}(W[-\ell],B[-\ell])|}{|P^{B[-\ell]}_{ 1}|} = \frac{q^{2m-2\ell -1} - q^{m-\ell+1} + q^{m-\ell} -1}{q-1} =\\ 
\displaystyle \frac{(q^{m-\ell}-1)(q^{m-\ell +1} +1)}{q-1}
\end{array},$$
we have 
$$\begin{array}{c}
\displaystyle \frac{|\text{O}(W,B)|}{|P^B_{ k+\ell}|} = \frac{\displaystyle \prod_{i = m-k-\ell }^{m-1} (q^{i+1} -1) \cdot Q_{i} }{(q^{k+\ell} -1) \cdot (q^{k+\ell -1} -1) \dots (q-1)} =\\
\\
\displaystyle  {m \choose k+\ell }_q \cdot \prod_{i = m-k-\ell }^{m-1} Q_i.
\end{array}$$
Similarly, we have
$$
\frac{|\text{O}(W[-\ell],B[-\ell])|}{|P^{B[-\ell]}_{ k}|} = \frac{ \displaystyle  \prod_{i = m-k-\ell}^{m-\ell-1} (q^{i+1}-1) Q_i  }{(q^{k }-1) \dots (q-1)}.$$

We also have
$$\begin{array}{c}
\displaystyle |\text{O} (W[-k-\ell], B[-k-\ell])| = \\
\\
\displaystyle 2 q^{(m-k-\ell) (m-k-\ell- 1)}\cdot (q^{m-k-\ell} -1) \prod_{j =1 }^{m-k-\ell-1} (q^{2j} -1) =\\
\\
\displaystyle 2 q^{(m-k-\ell) (m-k-\ell- 1)} \cdot (q^{m-k-\ell} -1) \prod_{j =1 }^{m-k-\ell-1} (q^{j} -1) (q^j +1).
\end{array}$$

Again, combining with \rref{OddLHSCORStateFirst}, we can reduce the claimto
\beg{SplitCaseFirstWriteOutDimsEven}{\begin{array}{c}
2 Q_{2m-\ell-1} \dots Q_1=\\[1ex]
\resizebox{0.85\textwidth}{!}{$\displaystyle \sum_{k=0}^{m-\ell} q^{(m-k-\ell) (m-k-\ell -1)} {m \choose k+\ell}_q 
\prod_{i= m-k -\ell}^{m-1} Q_i \frac{\displaystyle (q^{m-\ell}-1) \prod_{j=1}^{m-\ell -1}(q^j -1) Q_j}{(q^k -1) \dots (q-1)}$}
\end{array}}
Again, we may divide both
sides of \rref{SplitCaseFirstWriteOutDimsEven} by $2 Q_{m-\ell-1} \dots Q_1$, giving
\beg{Qm+1DotsQ2m-1+1DimsEvenSplit}{\begin{array}{c}
Q_{2m-\ell-1} \dots Q_{m-\ell} =\\[1ex] 
\displaystyle \sum_{k=0}^{m-\ell}  q^{(m-k-\ell) (m-k-\ell-1)} \cdot {m \choose k+\ell}_q 
\cdot \prod_{i = m-k-\ell}^{m-1} Q_i \cdot \prod_{j=k +1}^{m-\ell} (q^j -1).
\end{array}}
Rewriting ${m \choose k + \ell}_q = {m \choose m-k-\ell}_q$, this statement follows directly from applying Lemma \ref{CombinatorialLemma} to
$ r = m-1$, $p = m-\ell $, $b= 0$,
and substituting $a = m-k-\ell$.

\vspace{3mm}

{\em Case 2B:} Suppose $B$ has a non-trivial anisotropic part, i.e. $h_W = m -1$.
In this case, combining \rref{GORecursion} and \rref{PointsOfQuadricInPn-1}, noting first that we may first re-write this case of \rref{PointsOfQuadricInPn-1} as 
$$\begin{array}{c}
\displaystyle \frac{|\text{O}(W[-\ell],B[-\ell])|}{|P^{B[-\ell]}_{ 1}|}= \frac{q^{2m-2\ell -1} + q^{m-\ell+1} - q^{m-\ell} -1}{q-1} =\\ 
\displaystyle \frac{(q^{m-\ell} +1)(q^{m-\ell +1} - 1)}{q-1},
\end{array}$$
we have 
$$\begin{array}{c}
\displaystyle \frac{|\text{O}(W,B)|}{|P^B_{ k+\ell}|}= \frac{\displaystyle \prod_{i = m-k-\ell +1 }^{m} (q^{i-1} -1) \cdot Q_{i} }{(q^{k+\ell} -1) \cdot (q^{k+\ell -1} -1) \dots (q-1)} =\\
\\
\displaystyle  {m -1 \choose k+\ell }_q \cdot \prod_{i = m-k-\ell+1 }^{m} Q_i.
\end{array}$$
Similarly, we have
$$
\frac{|\text{O}(W[-\ell],B[-\ell])|}{|P^{B[-\ell]}_{k}|} = \frac{ \displaystyle  \prod_{i = m-k-\ell +1}^{m-\ell} (q^{i-1}-1) Q_i  }{(q^{k }-1) \dots (q-1)}.$$

We also have
$$\begin{array}{c}
\displaystyle |\text{O} (W[-k-\ell],B[-k-\ell])| = \\
\vspace{1mm}
\displaystyle 2 q^{(m-k-\ell) (m-k-\ell- 1)}\cdot (q^{m-k-\ell} +1) \prod_{j =1 }^{m-k-\ell-1} (q^{2j} -1) =\\
\vspace{1mm}\displaystyle 2 q^{(m-k-\ell) (m-k-\ell- 1)} \cdot (q^{m-k-\ell} +1) \prod_{j =1 }^{m-k-\ell-1} (q^{j} -1) (q^j +1).
\end{array}$$

Again, combining with \rref{OddLHSCORStateFirst}, we can reduce the statement
to
\beg{SplitCaseFirstWriteOutDimsEven}{\resizebox{0.9\textwidth}{!}{$\begin{array}{c}
2 Q_{2m-\ell -1} \dots Q_1=\\[1ex]
\displaystyle \sum_{k=0}^{m-\ell -1} q^{(m-k-\ell) (m-k-\ell -1)} \cdot {m-1 \choose k+\ell}_q 
\prod_{i= m-k -\ell +1}^{m} Q_i \frac{\displaystyle Q_{m-\ell } \prod_{j=1}^{m-\ell -1}(q^j -1) \cdot Q_j}{(q^k -1) \dots (q-1)}
\end{array}$}}
Similarly as in the previous cases, we may divide both
sides of \rref{SplitCaseFirstWriteOutDimsEven} by $2 Q_{m-\ell } \dots Q_1$, giving
\beg{Qm+1DotsQ2m-1+1DimsEvenSplit}{\begin{array}{c}
Q_{2m-\ell-1} \dots Q_{m-\ell +1} =\\[1ex] 
\resizebox{0.85\textwidth}{!}{$\displaystyle \sum_{k =0}^{m-\ell-1}  q^{(m-k-\ell) (m-k-\ell-1)} \cdot {m -1 \choose k+\ell}_q 
\cdot \prod_{i = m-k-\ell +1}^{m} Q_i \cdot \prod_{j=k +1}^{m-\ell-1} (q^j -1).$}
\end{array}}
Rewriting ${m-1 \choose k + \ell}_q = {m -1\choose m-k-\ell -1}_q$, this statement follows directly from applying Lemma \ref{CombinatorialLemma} to
$ r = m $, $p = m-\ell -1 $, $b= 2$,
and substituting $a = m-k-\ell -1$.

\end{proof}

\vspace{5mm}

\subsection{The proof of Lemma \ref{CombinatorialLemma}}\label{CombinLemmaProofSubsect}

The proof of Lemma \ref{CombinatorialLemma} proceeds by induction.
The argument is perhaps slightly unusual due to the fact that our formula does
not reduce well at $q \rightarrow 1$ and is therefore not a ``quantization" of a
classical formula.

\begin{proof}[Proof of Lemma \ref{CombinatorialLemma}]
To prove \rref{CombLemmaGenStat}, we begin by rewriting
\beg{FirstStepCombLemma}{\begin{array}{c}
Q_{r+p} \dots Q_{b + p} =\\[1ex]
\displaystyle Q_r \dots Q_b + \sum_{k = b}^{r} (Q_{p+k} - Q_{k}) \cdot
\prod_{j=k+1}^{r} Q_j \cdot \prod_{j' =b}^{k-1} Q_{j' +p}. 
\end{array}}

Now, for each $b \leq k \leq r$, we have
$$Q_{p+k} - Q_k = q^k \cdot (q^p -1),$$
so we may rewrite \rref{FirstStepCombLemma} as 
\beg{Step0CombinArguLemma}{\begin{array}{c}
Q_{r+p } \dots Q_{b+p} = \\[1ex]
\displaystyle
Q_r  \dots Q_b + \sum_{k=b}^{r} q^k (q^p -1) \prod_{j=k+1}^{r} Q_j \cdot \prod_{j' =b}^{k-1} Q_{j' +p}. 
\end{array}}

Our goal is now to process each term of the right hand side of \rref{Step0CombinArguLemma}
to convert, step by step, the highest appearing $Q_{j' + p}$ factor into the
next smallest $Q_j$ not yet appearing. In the statement of Lemma \ref{CombinatorialLemma},
all terms consist of multiples of products of the form
$$Q_r \cdot Q_{r-1} \dots Q_{b+a+1} \cdot Q_{b+a},$$
so we cannot skip any $Q_j$'s in the process. We make the following

\vspace{5mm}

\begin{claim}\label{ClaimCombinLemmaTranche}
For $b \leq k \leq r$, we have
\beg{TrancheTermsCombinArguLemma}{
\begin{array}{c}
\displaystyle q^k (q^p -1) \prod_{j = k+1}^r Q_j \cdot \prod_{j' = b}^{k-1} Q_{j'+p}=\\
\\
\displaystyle \sum_{a=1}^{k-b+1} \left( \sum_{a+b-1 \leq \ell_1 \leq \dots \leq \ell_a = k} q^{\ell_1 + \dots + \ell_a} \right)
\cdot  \prod_{i = p-a+1}^p (q^i -1) \cdot \prod_{j = b+a}^r Q_j
\end{array}}
\end{claim}

\vspace{5mm}

\begin{proof}[Proof of Claim \ref{ClaimCombinLemmaTranche}]
We will proceed inductively, step by step, converting each factor $Q_{j' +p}$ in a term
of the previous step's reduction of \rref{TrancheTermsCombinArguLemma}, starting 
with the largest appearing $j'$,
into a sum of the next lower $Q_j$ not yet appearing, with
the appropriate error term of $q^j$ multiplied by
a factor $(q^{j'+p-j} -1)$. This process will terminate
in $k-b$ steps (we are already done when $k = b$).

\vspace{5mm}

The induction hypothesis is that after $n$ steps, we will have reduced
\rref{TrancheTermsCombinArguLemma} to
\beg{nthinductivestepCombinArguL}{\resizebox{0.9\textwidth}{!}{$\displaystyle
\sum_{a=1}^{n+1} \left( \sum_{a-n+k-1 \leq \ell_1 \leq \dots \leq \ell_a = k} q^{\ell_1 + \dots + \ell_a} \right)
\cdot  \prod_{i = p-a+1}^p (q^i -1) \cdot \prod_{j = k -n +a }^r Q_j \prod_{j' = b}^{k - n-1} Q_{j' +p}$}}

\vspace{5mm}

Let us describe the first step of this process
for \rref{TrancheTermsCombinArguLemma}. The largest appearing $j'$ is $j' = k-1$. 
The next lower $Q_j$ factor not yet appearing is for $j = k$. Therefore, this step
uses the replacement
$$Q_{k-1+p} = Q_k + q^k (q^{p-1} -1).$$
This gives
$$
\begin{array}{c}
\displaystyle q^k (q^p-1) \prod_{j = k}^r Q_j \cdot \prod_{j' = b}^{k-2} Q_{j'+p} +\\
\\
\displaystyle q^{k+ k} (q^p-1) (q^{p-1}-1)  \prod_{j = k+1}^r Q_j \cdot \prod_{j' = b}^{k-2} Q_{j'+p} 
\end{array}
$$
proving \rref{nthinductivestepCombinArguL} at $n = 1$.

\vspace{5mm}

Suppose \rref{nthinductivestepCombinArguL} holds at Step $n$. We now need to 
perform Step $(n+1)$. For $1 \leq a \leq n+1$, consider the term
\beg{SingleTermOfStepnCombinArgu}{\resizebox{0.9\textwidth}{!}{$\displaystyle \left( \sum_{a-n+k-1 \leq \ell_1 \leq \dots \leq \ell_a = k} q^{\ell_1 + \dots + \ell_a} \right)
\cdot  \prod_{i = p-a+1}^p (q^i -1) \cdot \prod_{j = k -n +a }^r Q_j \cdot \prod_{j' = b}^{k - n-1} Q_{j' +p},
$}}
of \rref{nthinductivestepCombinArguL}.

\vspace{3mm}

The highest occuring $Q_{j'+p}$ is at $j' = k-n-1$.
The next lower $Q_j$ factor not yet appearing is for $j = k-(n+1)+a$. Therefore, in this term, we must use the replacement
$$Q_{k-n-1 +p} = Q_{k-(n+1)+a} + q^{k-(n+1) +a} \cdot (q^{p-a} -1).$$
This reduces \rref{SingleTermOfStepnCombinArgu} to
\beg{SingleInductiveTermNToN+1CombL}{\resizebox{0.9\textwidth}{!}{$\begin{array}{c}
\displaystyle \left( \sum_{a-n+k-1 \leq \ell_1 \leq \dots \leq \ell_a = k} q^{\ell_1 + \dots + \ell_a} \right)
\cdot  \prod_{i = p-a+1}^p (q^i -1) \cdot \prod_{j = k -(n+1) +a }^r Q_j \cdot \prod_{j' = b}^{k - n-2} Q_{j' +p}+\\
\\
+\displaystyle \left( \sum_{a-n+k-1 = \ell_0 \leq \ell_1 \leq \dots \leq \ell_a = k} q^{\ell_0+ \ell_1 + \dots + \ell_a} \right)
\cdot  \prod_{i = p-a}^p (q^i -1) \cdot \prod_{j = k -n +a }^r Q_j \cdot \prod_{j' = b}^{k - n-2} Q_{j' +p}
\end{array}$}}

These terms appear in the $(n+1)$th inductive step; the first one occurs in
the expression \rref{nthinductivestepCombinArguL} with $n$ replaced by $n+1$ with
no reindexing of $a$ or $\ell_i$, and the second one occurs after replacing
$a$ by $a+1$ and shifting $\ell_{0} \leq \dots \leq \ell_a$ to $\ell_{1}\leq \dots \leq \ell_{a+1}$.

Therefore, we may proceed inductively, and at Step $n = k-b$, we obtain the reduction \rref{TrancheTermsCombinArguLemma}.

\end{proof}

\vspace{5mm}

Recombining the terms \rref{TrancheTermsCombinArguLemma}
according to \rref{Step0CombinArguLemma}, we get
\beg{AlmostExpCompLemmaStat}{\begin{array}{c}
Q_{r+p} \dots Q_{b + p} =\\[1ex]
\displaystyle 
\sum_{a = 0}^p  \left(\sum_{a+b -1 \leq \ell_1 \leq \dots \leq \ell_a \leq r} q^{\ell_1 + \dots + \ell_a}\right) \cdot
\prod_{i = p -a+1}^p (q^k -1) \cdot \prod_{j= b+a}^r Q_j 
\end{array}
}
(note that the $a= 0$ term arises from the single term $Q_r \dots Q_b$ in \rref{Step0CombinArguLemma}).

\vspace{3mm}

Finally, we compute
$$\begin{array}{c}
\displaystyle \sum_{a+b -1 \leq \ell_1 \leq \dots \leq \ell_a \leq r} q^{\ell_1 + \dots + \ell_a} = 
 q^{a(a+b-1)} \sum_{0 \leq \ell_1 \leq \dots \leq \ell_a \leq r-a-b+1} q^{\ell_1+ \dots + \ell_q}= \\
\\
\displaystyle q^{a(a+b-1)} \cdot {r-b+1 \choose a}_q,
\end{array}$$
by the Gaussian binomial coefficient theorem. Plugging this into \rref{AlmostExpCompLemmaStat}
gives \rref{CombLemmaGenStat}.

\end{proof}

\vspace{5mm}

\section{The proof of Theorem \ref{TheoremGenuine}}\label{ProofSection}

In this section, we conclude the proof of Theorem \ref{TheoremGenuine}
using an inductive argument that proceeds completely similarly in both stable ranges. 

\begin{proof}
We begin with \rref{SymTheoremGenuine}. Fix a symplectic space $V$.
We proceed by induction on $h_W$. Suppose the statement of \rref{SympStableEndDecompThm}
holds for every choice of $(W,B)$ with $h_W < m$. Now
consider a choice of $(W,B)$ with $h_W=m$.

For every $1 \leq k \leq m$, the induction hypothesis applies
to $\omega [ V\otimes W[-k]]$. Write, for $1 \leq \ell \leq m$,
$$Z_\ell = \bigoplus_{\rho \in \widehat{\text{O} (W[-\ell], B[-\ell])}} \eta^V_{W[-\ell], B[-\ell]} (\rho)
\otimes \rho$$
as an $\text{Sp}(V) \times \text{O}(W[-\ell], B[-\ell])$-representation. We have
$$\C \text{O}(W[-\ell], B[-\ell]) = \text{End}_{\text{Sp}(V)} (Z_\ell) = \text{End}_{\text{Sp}(V)}^{top} (\omega [ V\otimes W[-\ell]]),$$
and
$$\text{Res}_{\text{Sp}(V)} (\omega [V\otimes W[-\ell]]) \cong \bigoplus_{k=0}^{h_W-\ell} \frac{|\text{O}(W[-\ell],B[-\ell])|}{|P_{k}^{B[-\ell]}|}
\text{Res}_{\text{Sp}(V)} (Z_{k+\ell}).$$

First, we claim that as $\text{Sp}(V)$-representations, a sum of $|\text{O}(W,B)/P_\ell^B|$ copies
of $\text{Res}_{\text{Sp}(V)} (Z_\ell)$ appears in $\text{Res}_{\text{Sp}(V)} (\omega [ V\otimes W])$.
Now this holds since by the induction hypothesis, we have
\beg{UnEvenForFinalIndSymp}{\begin{array}{c}
\text{Hom}_{\text{Sp}(V)} (\omega[ V\otimes W[-\ell]], \omega [ V\otimes W])= \\[1ex]
\displaystyle \bigoplus_{k=0}^{h_W-\ell} \frac{|\text{O}(W[-\ell], B[-\ell])|}{|P^{B[-\ell]}_k|} \cdot
\text{Hom}_{\text{Sp}(V)} ( Z_{k+\ell}, \omega [ V\otimes W]).
\end{array}}
Applying Proposition \ref{HomUevenDimsCor}, the dimension of the left hand side of
\rref{UnEvenForFinalIndSymp} can be re-expressed
as
$$\sum_{k=0}^{h_W -\ell} \frac{|\text{O} (W[-\ell], B[-\ell])|}{|P_k^{B[-\ell]}|} \cdot
\frac{|\text{O}(W,B)|}{|P_{k+\ell}^B|} \cdot |\text{O} (W[-k-\ell], B[-k-\ell])|$$
giving a linear system of equations for the dimensions of the $Hom$-spaces
$\text{Hom}_{\text{Sp} (V)} (Z_{k+\ell}, \omega [V\otimes W])$, giving
$$\text{dim} (\text{Hom}_{\text{Sp} (V)} (Z_\ell, \omega [ V\otimes W])) = \frac{|\text{O}(W,B)|}{|P_{\ell}^B|}
|\text{O}(W[-\ell], B[-\ell])|.$$
This gives the claim since this $\text{Hom}$-space is a free module over
the endomorhism algebra
$\text{End}_{\text{Sp}(V)} (Z_k) = \C \text{O}(W[-k], B[-k])$, and we can
consider
$$\text{Hom}_{\text{Sp}(V)\times \text{O}(W[-k], B[-k])} (Z_k, \omega [V\otimes W]).$$
Since we have
\beg{RestrictSlightly}{\text{Res}_{\text{Sp}(V)\times \text{O}(W[-k],B[-k])} (\omega [V\otimes W] )
= \omega [ V\otimes W[-k]] \otimes (\C V)^{\otimes k},}
where $\text{O}(W[-k], B[-k])$ acts trivially on the $\C V$ factors,
the action of $\text{O}(W[-k], B[-k])$ on each copy of $Z_k$ is preserved and therefore, $Z_k$
occurs as a summand of \rref{RestrictSlightly} with mutliplicity $|\text{O}(W,B)/P_k^B|$.
Now, since we know that
$$\text{Res}_{\text{GL}_N (\F_q)} (\omega [ V\otimes W]) = (\C \F_q^N) \otimes\epsilon (\text{det}),$$
we get, by adjunction, summands
$$\bigoplus_{k=1}^{h_W} \bigoplus_{\rho \in \widehat{\text{O}(W[-k], B[-k])}} 
\eta^V_{W[-k], B[-k]} (\rho) \otimes \text{Ind}_{P_k^B} (\rho\otimes \epsilon (\text{det})),$$
as in the claimed expression \rref{EtaCorrThmDecomp}.
From the perspective of endomorphism algebras, we have
$$\bigoplus_{k=1}^{h_W} M_{|\text{O}(W,B)/P_k^B|} (\C \text{O}(W[-k], B[-k])) \subseteq \text{End}_{\text{Sp}(V)} (\omega [ V\otimes W]),$$
all occuring independently from the top part, since they specifically all arise from
terms appearing in $\text{End}_{\text{Sp}(V)} (\omega [V\otimes W[-k]])$.
Combining with Proposition \ref{GroupInTop} gives the $k=0$ term as well:
$$\bigoplus_{k=0}^{h_W} M_{|\text{O}(W,B)/P_k^B|} (\C \text{O}(W[-k], B[-k])) \subseteq \text{End}_{\text{Sp}(V)} (\omega [ V\otimes W]).$$
This must be an equality, however, since the dimensions match by Corollary \ref{Combinatoricsdimmatchell=0topfinal}.

\vspace{5mm}

Similarly, to prove part \rref{OrthoTheoremGenuine}, we fix $(W,B)$.
and proceed by induction on $N$. Suppose the statement of \rref{OrthoStableEndDecompThm}
holds for every symplectic space $V$ of dimension $2M$ with $M < N$. Now
consider a symplectic space $V$ of dimension $2N$.

For every $1 \leq k \leq N$, the induction hypothesis applies
to $\omega [ V[-k]\otimes W ]$. Write, for $1 \leq \ell \leq m$,
$$Y_\ell = \bigoplus_{\rho \in \widehat{\text{Sp}(V[-\ell])}}  \rho \otimes
\zeta^{W,B}_{V[-\ell]} (\rho)$$
as an $\text{Sp}(V[-\ell]) \times \text{O}(W, B)$-representation. We have
$$\C \text{Sp}(V[-\ell]) = \text{End}_{\text{O}(W,B)} (Y_\ell) = \text{End}_{\text{O}(W,B)}^{top} (\omega [ V[-\ell]\otimes W]),$$
and
$$\text{Res}_{\text{O}(W,B)} (\omega [V[-\ell]\otimes W ]) \cong \bigoplus_{k=0}^{h_W-\ell} \frac{|\text{Sp}(V[-\ell])|}{|P_{k}^{V[-\ell]}|}
\text{Res}_{\text{O}(W,B)} (Y_{k+\ell}).$$
Again, we begin by arguing that is $O(W,B)$-representations, a sum of
$|\text{Sp}(V)/P_\ell^V|$ copies
of $\text{Res}_{\text{O}(W,B)} (Y_\ell)$ appears in $\text{Res}_{\text{O}(W,B)} (\omega [ V\otimes W])$,
which holds by applying the induction hypothesis to get
\beg{OrthoFinalIndStepLinearSyst}{\begin{array}{c}
\text{Hom}_{\text{O}(W,B)} (\omega[ V[-\ell]\otimes W], \omega [ V\otimes W])= \\[1ex]
\displaystyle \bigoplus_{k=0}^{h_W-\ell} \frac{|\text{Sp}(V[-\ell])|}{|P^{V[-\ell]}_k|} \cdot
\text{Hom}_{\text{O}(W,B)} ( Y_{k+\ell}, \omega [ V\otimes W]).
\end{array}}
Applying Proposition \ref{HomUevenDimsCor}, the dimension of the left hand side of \rref{OrthoFinalIndStepLinearSyst}
can be re-expressed as
$$\sum_{k=0}^{h_W -\ell} \frac{|\text{Sp} (V[-\ell])|}{|P_k^{V[-\ell]}|} \cdot
\frac{|\text{Sp}(V)|}{|P_{k+\ell}^V} \cdot |\text{Sp}(V[-k-\ell])|$$
giving a linear system of equations for the dimensions of the $\text{Hom}$-spaces
$\text{Hom}_{\text{O} (W,B)} (Y_{k+\ell}, \omega [V\otimes W])$, giving
$$\text{dim} (\text{Hom}_{\text{O}(W,B)} (Y_\ell, \omega [ V\otimes W])) = \frac{|\text{Sp}(V)|}{|P_{\ell}^V|}|\text{Sp}(V[-\ell])|.$$
This gives the claim since this $\text{Hom}$-space is a free module over the endomorphism algebra
$\text{End}_{\text{O}(W,B)} (Y_k) = \C \text{Sp}(V[-k])$, and we can consider
$$\text{Hom}_{\text{Sp}(V[-k])\times \text{O}(W,B)} (Y_k, \omega [V\otimes W]).$$
Since we have
\beg{RestrictSlightly}{\text{Res}_{\text{Sp}(V[-k])\times \text{O}(W,B)} (\omega [V\otimes W] )
= \omega [ V[-k]\otimes W] \otimes (\C W)^{\otimes k},}
where $\text{Sp}(V[-k])$ acts trivially on the $\C W$ factors,
the action of $\text{Sp}(V[-k])$ on each copy of $Y_k$ is preserved and therefore, $Y_k$
occurs as a summand of \rref{RestrictSlightly} with mutliplicity $|\text{Sp}(V)/P_k^V|$.
Again, by adjunction, we get summands
$$\bigoplus_{k=1}^{N} \bigoplus_{\rho \in \widehat{\text{Sp}(V[-k])}} 
\text{Ind}_{P_k^V} (\rho\otimes \epsilon (\text{det}))\otimes \zeta^{W, B}_{V[-k]} (\rho) ,$$
as in the claimed expression \rref{ZetaCorrThmDecomp}.
From the perspective of endomorphism algebras, we have
$$\bigoplus_{k=1}^{N} M_{|\text{Sp}(V)/P_k^V|} (\C \text{Sp}(V[-k])) \subseteq \text{End}_{\text{O}(W,B)} (\omega [ V\otimes W]),$$
all occuring independently from the top part, since they specifically all arise from
terms appearing in $\text{End}_{\text{O}(W,B)} (\omega [V[-k]\otimes W])$.
As before, combining with Proposition \ref{GroupInTop} gives that, in fact,
$$\bigoplus_{k=0}^{N} M_{|\text{Sp}(V)/P_k^V|} (\C \text{Sp}(V[-k])) \subseteq \text{End}_{\text{O}(W,B)} (\omega [ V\otimes W])$$
This must be an equality, again, since the dimensions match by Corollary \ref{Combinatoricsdimmatchell=0topfinal}.
\end{proof}

\appendix

\section{An Explicit Composition Computation}

\vspace{5mm}
Fix $t\in  \F_q^\times$.
In this appendix, we return to the context
of $2$-dimensional symplectic spaces, and verify a relation between the endomorphisms
of the oscillator representation implied by their correspondence with elements of the group algebra
$\C \text{SL}_2 (\F_q)$.
Now that we have calculated the elements of $(\C \F_q^2 \otimes W)^{\text{O}(W,B)}$ corresponding
to the usual generators of $SL_2 (\F_q)$, they must.
Fix a scalar $t\in  \F_q^\times$. We have the following relation of matrices
$$ \begin{pmatrix}
0 & t\\
-1/t & 2
\end{pmatrix} \cdot
\begin{pmatrix}
t & 0\\
0 & t^{-1}
\end{pmatrix}=
\begin{pmatrix}
1 & 0\\
2/t & 1
\end{pmatrix}\cdot \begin{pmatrix}
0 & 1\\
-1 & 0
\end{pmatrix}$$
in $SL_2 (\F_q)$.
The purpose of this appendix is to complete the calculation that
\beg{galphaisgammabeta}{ g_t \star \alpha_t = \gamma_{2/t} \star \beta.}
The composition $g_t \star \alpha_t$ is 
$\frac{\epsilon(t)^2}{q^n (-q)^{n/2}} = \frac{1}{q^n (-q)^{n/2}}$ times the sum over all choices of
$y^+, y^-, z \in W$ of terms
\beg{galphawriteout}{
\psi (\frac{t}{2} B(z,z) - \frac{t+1}{2(t-1)} B (y^+ , y^-)) \cdot (z, tz) \star (y^+, y^-).
}
Writing out
$$(z, tz) \star (y^+, y^-) = \psi (\frac{1}{2} (B (z, y^-) - t \cdot B (z, y^+))),$$
each term \rref{galphawriteout} can be simplified to
the pair of vector $(y^+ + z, y^- + tz)$ multiplied by the coefficient
$$\psi (\frac{t}{2} B (z,z) -\frac{t+1}{2 (t-1)} B (y^+, y^-) + \frac{1}{2} B (z, y^-) - \frac{t}{2} B (z, y^+)).$$
By considering
$$ -\frac{t+1}{2 (t-1)} = \frac{1}{2} - \frac{t}{t-1}, \hspace{10mm} -\frac{t}{2} = \frac{t}{2}-t,$$
this can be rewritten as
$$\psi (\frac{1}{2} B (y^+ + z,  y^- + tz) - \frac{t}{t-1} B (y^+, y^- + (a-1)z)).$$
Substituting $u = y^+ + z$, $v = y^- + tz$ gives
$$\psi ( \frac{1}{2} B (u,v) - \frac{t}{t-1} B (u-z, v-z)).$$
Therefore, we have reduced $g_t \star \alpha_t$ to
\beg{galphaalmostdone}{
\frac{1}{q^n (-q)^{n/2}} \sum_{z,u,v \in W} \psi (\frac{1}{2} B (u,v) - \frac{t}{t-1} B (u-z, v-z)) \cdot (u,v).
}

Writing
$$B (u-z, v-z) = B (u,v) - B (u+ v, z) + B (z,z),$$
we may ``complete the square" by noticing that
$$\begin{array}{c}
- B (u+v, z) + B(z,z) =\\[1ex]
\displaystyle B (z - \frac{u+v}{2}, z- \frac{u+v}{2}) - B (\frac{u+v}{2}, \frac{u+v}{2}).
\end{array}$$
Substituting variables using $w = z- (u+v)/2$, putting the terms together, we get
$$\begin{array}{c}
\displaystyle B (u-z, v-z) = B(u,v) + B(w,w) - B (\frac{u+v}{2}, \frac{u+v}{2}) = \\
\\
\displaystyle B (w,w) - B (\frac{u-v}{2}, \frac{u-v}{2}).
\end{array}$$

Therefore, \rref{galphaalmostdone}
reduces to
\beg{galphafinal}{
\begin{array}{c}
g_t \star \alpha_t = \\[1ex]
\resizebox{0.85\textwidth}{!}{$\displaystyle\frac{1}{q^n (-q)^{n/2}}\sum_{w, u,v \in W} \psi (\frac{1}{2} B(u,v) - \frac{t}{t-1} (B(w,w) - B (\frac{u-v}{2}, \frac{u-v}{2}))) \cdot (u,v) =$}\\
\\
\resizebox{0.85\textwidth}{!}{$\displaystyle
\frac{K (-t/ (t-1))}{q^n (-q)^{n/2}} \sum_{u,v \in W} \psi (\frac{1}{2} B (u,v) + \frac{t}{t-1} B (\frac{u-v}{2}, \frac{u-v}{2})) \cdot (u,v).$}
\end{array}
}

\vspace{5mm}

Now let us consider the other side of \rref{galphaisgammabeta}.
The composition $\gamma_{2/t} \star \beta$ is $\frac{1}{ (-q)^{n/2} K (-t/4) K(1)}$ 
times the sum over all choices of $y^+, y^-, z\in W$ of terms
$$
\psi (-\frac{t}{4} B (z,z) + \frac{1}{4} (B (y^+, y^+) + B (y^-, y^-))) \cdot (z, 0 ) \star (y^+, y^-)
$$
Writing out
$$(z, 0 ) \star (y^+, y^-) = \psi (\frac{1}{2} B (z, y^-)) \cdot (z + y^+, y^-),$$
this term is the pair of vectors $(z+ y^+, y^-)$ multiplied by the coefficient
$$\begin{array}{c}
\displaystyle
 \psi (-\frac{t}{4}B (z,z) +\frac{1}{4} (B (y^+, y^+) + B (y^-, y^-)) +\frac{1}{2} B (z, y^-)) = \\
\\
\displaystyle
\psi (-t \cdot B(\frac{z}{2},\frac{z}{2}) +B (\frac{y^+ - y^-}{2}, \frac{y^+ - y^-}{2}) + \frac{1}{2} B(z+ y^+, y^-) )
\end{array}$$
Substituting variables $u = z+ y^+$, $v = y^-$, $w= z/2$ we reduce $\gamma_{2/t} \star \beta$ to
the coefficient
$\frac{1}{ (-q)^{n/2} K (-t/4) K(1)}$ times
$$
\sum_{ u, v, w\in W } \psi (-t B (w,w) +  B (\frac{u-v}{2}-w, \frac{u-v}{2}-w) + \frac{1}{2} B(u,v)) \cdot (u,v).
$$ 
Writing
$$\begin{array}{c}
\displaystyle B (\frac{u-v}{2}-w, \frac{u-v}{2}-w) = \\
\\
\displaystyle B (\frac{u-v}{2}, \frac{u-v}{2})-  2 \; B (\frac{u-v}{2}, w) + B (w,w),
\end{array}$$
we have
$$\begin{array}{c}
\displaystyle
-t B(w,w) +   B (\frac{u-v}{2}-w, \frac{u-v}{2} -w) =\\
\\
\displaystyle
-(t-1) \cdot B (w,w ) - 2\; B (\frac{u-v}{2}, w) + B (\frac{u-v}{2}, \frac{u-v}{2}).
\end{array}
$$
Completing the square gives
$$\begin{array}{c}
\displaystyle B(w,w) +\frac{2}{t-1} B (\frac{u-v}{2}, w) =\\
\\
\displaystyle B (w+ \frac{u-v}{2(t-1)},  w+ \frac{u-v}{2(t-1)}) - \frac{1}{(t-1)^2} B (\frac{u-v}{2}, \frac{u-v}{2}).
\end{array}
$$
Replacing variables $x= w+ \frac{1}{2(t-1)} (u-v)$ gives
$$\begin{array}{c}
\displaystyle -t B (w,w) +  B (\frac{u-v}{2}-w, \frac{u-v}{2}-w) + \frac{1}{2} B(u,v)=\\
\\
\displaystyle -(t-1)B (x,x) + (1 + \frac{t-1}{(t-1)^2}) B (\frac{u-v}{2}, \frac{u-v}{2}) + \frac{1}{2} B (u,v)
=\\
\\
\displaystyle -(t-1)B (x,x) + \frac{t}{t-1} B (\frac{u-v}{2}, \frac{u-v}{2}) + \frac{1}{2} B (u,v)
\end{array}$$

Thus, $\gamma_{2/t} \star \beta$ is the factor $ \frac{1}{(-q)^{n/2} K(-t/4) K (1)}$ times
$$\begin{array}{c}
\displaystyle
\resizebox{0.9\textwidth}{!}{$\displaystyle\sum_{u,v, x \in W} \psi (-(t-1) B(x,x) + \frac{t}{t-1} B (\frac{u-v}{2}, \frac{u-v}{2}) + \frac{1}{2}B(u,v))
\cdot (u,v) =$}\\
\\
\displaystyle
K(-(t-1)) \cdot \sum_{u,v \in W} \psi (\frac{t}{t-1} B (\frac{u-v}{2}, \frac{u-v}{2}) + \frac{1}{2}B(u,v))
\cdot (u,v).
\end{array}
$$

This agrees with our above calculation of $g_t \star \alpha_t$ in \rref{galphafinal}, up to a constant.
It remains to check that the constants precisely agree, i.e.
\beg{FinalConstantsAppAgree}{\frac{K (-(t-1))}{(-q)^{n/2} K(-t/4) K (1)} = \frac{K (-t/ (t-1))}{q^n(-q)^{n/2}}.}
Recalling \rref{Konstant}, first note that since 
$$K(c) = (-1)^{n (\ell+1)} \cdot \text{disc}(B) \cdot q^{n/2} \cdot \epsilon_q (c)^n \cdot \epsilon_q (-1)^{n/2}$$
only depends on $\epsilon_q (c)$, we have
$K (-t/4) = K(-t)$.
We can therefore simplify \rref{FinalConstantsAppAgree} to
$$ q^{n} \cdot K (-(t-1)) = K (-t/ (t-1)) \cdot K (-t) \cdot K (1).$$
Next, the signs, i.e. the factors $(-1)^{n(\ell+1)} \text{disc}(B)$ in each $K$ factor
will cancel, since both the left and right hand side have and odd number of $K$ factors.
Further, collecting factors, both sides have a factor of $q^n(-q)^{n/2}$, which we may factor out.
This reduces the claim to
$$\epsilon_q( - (t-1))^{n} \cdot \epsilon (-1)^{n/2} = \epsilon_q (-t/(t-1))^n \epsilon_q (-t)^n \epsilon_q (-1)^{3n/2}.$$
Dividing both sides by $\epsilon_q (-1)^{n/2}$ and collecting terms gives
$$ \epsilon_{q} (-(t-1))^n = \epsilon_q (\frac{-t}{t-1} \cdot (-t) \cdot (-1))^n,$$
which holds, since $\epsilon_q (-(t-1)) = \epsilon_q(-1/(t-1))$.

\end{document}